\documentclass{article}
\usepackage{comment}
\usepackage{graphics}	
\usepackage[utf8]{inputenc}
\usepackage[english]{babel}										%
\usepackage{amsmath}
\usepackage{tikz}
\usetikzlibrary{cd}
\usepackage{enumerate}
\usepackage{bbm}
\usepackage{tikz-cd}
\tikzcdset{scale cd/.style={every label/.append style={scale=#1},
    cells={nodes={scale=#1}}}}
\usepackage{amssymb}
\usepackage{mathrsfs}
\tikzcdset{scale cd/.style={every label/.append style={scale=#1},
    cells={nodes={scale=#1}}}}
\usepackage{amsthm}
\newtheorem{theorem}{Theorem}
\newtheorem{lemma}[theorem]{Lemma}
\newtheorem{prop}[theorem]{Proposition}
\newtheorem{ex}[theorem]{Example}
\newtheorem{deff}[theorem]{Definition}

\newtheorem{rmk}[theorem]{Remark}
\newcommand{\Res}{\text{Res}}

\newcommand{\hh}{\mathcal{H}}
\newcommand{\GG}{\mathcal{G}}
\newcommand{\ff}{\mathcal{F}}
\newcommand{\EE}{\mathcal{E}}
\newcommand{\cala}{\mathcal{A}}

\newcommand{\Spec}{\text{Spec}}

\newcommand{\im}{\text{im }}

\newcommand{\Hom}{\text{Hom}}

\newcommand{\Z}{\mathbb{Z}}
\newcommand{\C}{\mathbb{C}}

\newcommand{\oo}{\mathcal{O}}

\newcommand{\End}{\mathrm{End}}

\newcommand{\id}{\mathrm{id}}

\newcommand{\Ob}{\mathsf{Ob}}

\newcommand{\Cech}{\textit{\v{C}}}
\begin{document}
\author{Ville Nordström}
\title{On the map induced on Hochschild homology of matrix factorization categories by the inclusion of a divisor}
\date{}
\maketitle
\begin{abstract}Given a smooth variety $X$ over $\C$, a smooth divisor $i:Y\hookrightarrow X$ and a global function $f$ on $X$ which vanishes on $Y$ and on its critical locus we compute the map induced on Hochschild homology by the pushforward functor $i_*:D^b(Y)\to D^{abs}(MF(X,f))$ in terms of the Hochschild-Kostant-Rosenberg isomorphisms.
\end{abstract}

\section{Introduction}
For a smooth projective variety $X$ the classical Hirzebruch-Riemann-Roch theorem computes the Euler characteristic of a vector bundle in terms of its Chern character with values in singular cohomology. There is a version of Hirzebruch-Riemann-Roch that uses Hochschild homology instead of singular cohomology (see \cite{Cal2}, \cite{Markarian}, \cite{Ramadoss}). Key ingredients for this are Chern characters with values in Hochschild homology and the Mukai pairing. The HKR theorem, which states that for a non-singular variety $X$ we have
\begin{align}\label{HKR1}HH_i(X)\cong \bigoplus_{q-p=i}H^q(X,\Omega_X^p),\end{align}
 relates the two different settings. Using dg categories, an analogous picture has been developed for categories of matrix factorizations $MF(X,f)$ (see for example \cite{Kim-Pol}, \cite{Pol-Vaintrob}, \cite{Kim1}, \cite{Efimov}). When the function $f$ is assumed to vanish on its critical locus there is a HKR type theorem for the $\Z/2$-graded Hochschild homology of the dg category of matrix factorizations 
 \begin{align}\label{HKR2}HH_\bullet(MF(X,f))\simeq R\Gamma(X,(\Omega_X^\bullet,-df\wedge(-))).\end{align}
Given a morphism $\varphi:X\to Y$ between smooth projective varieties, the results in \cite{Ramadoss} and \cite{Cal2} allow us to understand the map in Hochschild homology induced by the pushforward functor $\varphi_*:D^b(\text{coh}Y)\to D^b(\text{coh} X)$ in terms of the HKR isomorphism and it is natural to try and understand this for matrix factorizations as well. In this paper we will consider the simplest case. Suppose $i:Y\hookrightarrow X$ is a smooth divisor of a smooth (not necessarily projective) variety $X$ and suppose $f\in H^0(X,\oo_X)$ is a regular function that satisfies $f|_Y=0$. Then we have a pushforward functor $D^b(\text{coh}Y)\to D^{abs}(MF(X,f))$ which lifts to a pseudo functor (see section \ref{pseudofunctorsbetweendgcats} for the definition of a pseudo functor) between their respective dg enhancements $\text{perf}_\Cech(Y)\to \text{vect}_\Cech(X,f)$ and so it induces a map in Hochschild homology 
\begin{align}\label{themap}
HH_\bullet(i_*):HH_\bullet(\text{perf}_\Cech(Y))\to HH_\bullet(\text{vect}_\Cech(X,f)).
\end{align}
Because the dg category of matrix factorizations is $\Z/2$-graded we only consider the groups above as $\Z/2$-graded vector spaces. We will denote by $D_{\Z/2}(\C)$ the derived category of $\Z/2$-graded vector spaces whose objects are $Z/2$-graded complexes of vector spaces and morphisms are chain maps where quasi isomorphisms have been inverted. The goal of this paper is to describe the map (\ref{themap}) in terms of the isomorphisms $(\ref{HKR1})$ and $(\ref{HKR2})$ above. We show that it can be computed in two steps. First, multiplication by the inverse Todd class of the line bundle $\oo_Y(-Y)$ realized in a Cech resolution of $\Omega^\bullet _Y$. Then, applying the connecting morphism $\delta$ coming from a short exact sequence of complexes of vector spaces obtained by taking Cech resolutions of the following short exact sequence of complexes of sheaves
$$0\to (\Omega_X^\bullet,-df\wedge(-))\to (\Omega_X^\bullet(\log Y),-df\wedge(-))\to (\Omega_Y^{\bullet-1},0)\to 0.$$
Our final formula is summarized in the following theorem.
\begin{theorem}
	There is a commutative diagram in $D_{\Z/2}(\C)$
	$$\begin{tikzcd}
		HH_\bullet(Y)\arrow{d}{\sim}\arrow{rrr}{HH_\bullet(i_*)}&&&HH_\bullet(X,f)\arrow{d}{\sim}\\
		\Cech(\mathscr{U},\Omega_Y^\bullet)\arrow{rr}{\bar{\wedge}Td(-Y)^{-1}}&&\Cech(\mathscr{U},\Omega_Y^\bullet)\arrow{r}{\delta}&\Cech(\mathscr{U},(\Omega_X^\bullet,-df\wedge-)).
	\end{tikzcd}$$
\end{theorem}
 
\subsection{Summary of paper}
Here is a summary of the paper. In chapter 2 we collect some facts about cdg categories and Hochschild homology of such. In section \ref{presheavesofcdg} we recall the notion of presheaves of cdg and dg categories. We introduce a notion of lax morphisms of presheaves of dg categories and explain how they give rise to a morphisms on Hochschild homology after passing to a Cech resolution.

In chapter 3 we recall how the Hochschild homology of $MF(X,f)$ is computed. We follow the computation in \cite{Efimov} closely. The only difference is that we use Cech instead of Godement resolutions as we find them better suited for computations.

In chapter 4 we explain the main problem in more detail and prove the theorem above. In section 4.5 we define a trace map similar to that in \cite{Kim1}, proposition 5.3, except we only consider Hochschild homology and not Cyclic homology. Our trace map is also different in that it uses Cech resolutions instead of Godement resolutions and its target is different.

In section \ref{maintheorem} we prove the theorem above. The main commutative diagram that goes into proving the theorem above is lemma \ref{diagram1}.

\subsection{Notation and conventions}\label{notation}
We work over $\C$ and all varieties, vector spaces and algebras will be over $\C$.

If $(\mathcal{F}^\bullet,d_{\mathcal{F}})$ is a complex of sheaves on a scheme $X$, and $\mathscr{U}$ is an affine open covering of $X$ then we make $Tot(\Cech{^\bullet}(\mathscr{U},\mathcal{F}^\bullet))$ into a complex by taking as our differential $d_{Cech}+\bar{d}_\mathcal{F}$ where $\bar{d}_\mathcal{F}$ is defined on $\alpha\in \Cech{^p}(\mathscr{U},\mathcal{F}^q)=(-1)^pd_\mathcal{F}(\alpha)$. 

If there is a multiplication defined $\mathcal{F}^\bullet\otimes \mathcal{G}^\bullet\to \mathcal{H}^\bullet$ we will use this to define a product 
$$\Cech(\mathscr{U},\ff^\bullet)\otimes \Cech(\mathscr{U},\GG^\bullet)\to \Cech(\mathscr{U},\hh^\bullet)$$
by the rule 
$$(f\cdot g)_{i_0\cdots i_{p+q}}:=f_{i_0\cdots i_p}g_{i_p\cdots i_{p+q}}$$
where $f\in \Cech^p(\mathscr{U},\ff^\bullet)$ and $g\in \Cech^q(\mathscr{U},\GG^\bullet)$.

We will denote by $D^b(\text{coh}X)$ the usual derived category of coherent sheaves on $X$ whose objects are bounded $\Z$-graded complexes. However, we will mostly be concerned with $\Z/2$-graded dg (or cdg)  categories. Categories that are usually $\Z$ graded, for example a dg enhancement $\mathscr{D}^b(\text{coh}(X))$ of $D^b(\text{coh}X)$ will be implicitly thought of as $\Z/2$-graded by forgetting the $\Z$-grading on the hom-spaces (but not on the objects!) and only remembering the parity of the degree of a morphism. 

\subsection{Acknowledgments}
I wish to thank my advisor Alexander Polishchuk for all his help with this project.

\section{Background on cdg categories, modules and Hochschild invariants}
Here, we recall some facts about cdg categories and Hochschild type invariants of such. We will mostly follow the conventions in \cite{Kim1} and also in \cite{Pol-Pol}. 

\subsection{Cdg categories and functors} Recall that a \textit{curved differential graded category}  $\mathscr{C}$(\textit{cdg category} for short) is a triple $(\mathscr{C},d,h)$ where $\mathscr{C}$ is a category enriched over $\Z/2$-graded vector spaces, $d$ is an assignment of a degree $1$ operator (called \textit{differential}) $d_{x,y}:\Hom_\mathscr{C}(x,y)\to\Hom_\mathscr{C}(x,y)[1] $ for every pair of objects $x$ and $y$ in $\Ob(\mathscr{C})$ and $h$ is the assignment of a degree $2$ endomorphism $h_x\in \End_\mathscr{C}(x)$ (called \textit{curvature}) for every object $x\in \Ob(\mathscr{C})$. This structure $d$ and $h$ must satisfy the following, $d$ is a derivation for the composition, $d_{x,x}(h_x)=0$ and $d_{x,y}^2(f)=h_y\circ f-f\circ h_x$
for all $f\in \Hom_\mathscr{C}(x,y)$. Here are some examples that will be important to us.

\begin{ex}
A dg category can be thought of as a cdg category with $h_x=0$ for all objects $x$.
\end{ex}

\begin{ex}\label{curvedalgebra}
Let $A$ be a graded algebra and let $h$ be an element of the center of $A$. Then $(A,0,h)$ is a cdg algebra with only one object. 	
\end{ex}

\begin{ex}
Let $Com_{cdg}(k)$ denote the cdg algebra whose objects are pairs $(C^\bullet,\delta_C)$ where $C^\bullet$ is a graded \text{vect}or space and $\delta_C:C^\bullet\to C^\bullet[1]$ is a degree 1 operator. We define $d_{C,C'}(f):=\delta_{C'}\circ f-(-1)^{|f|}f\circ \delta_C$ and $h_C:=\delta_C^2$. If $(C^\bullet,\delta_C)$ is an object of $Com_{cdg}(k)$ we define the shift $(C^\bullet[1],\delta_C[1]=-\delta_C)$ as usual.
\end{ex}

Let $\mathscr{C}$ and $\mathscr{C}'$ be two cdg categories. A functor $F:\mathscr{C}\to\mathscr{C}'$ of categories enriched over graded \text{vect}or spaces is called a \textit{quasi cdg functor} if
	$F_{x,y}\circ d_{x,y}=d_{F(x),F(y)}\circ F_{x,y}$ for all objects $x$ and $y$.
If in addition we have $F(h_x)=h_{F(x)}$ for all objects $x$ then we call $F$ a \textit{cdg functor}. 

\begin{rmk}What we just defined is sometimes called a strict functor but we omit the prefix strict since we will not need the notion of "non-strict" cdg functors.	 
\end{rmk}

\textit{Cdg (quasi-)modules} over a cdg category $\mathscr{C}$ are by definition (quasi) cdg functors $F:\mathscr{C}\to Com_{cdg}(k)$. If $F$ is a quasi functor we denote by $F[n]$ the quasi functor obtained by composing $[n]\circ F$. The collection of quasi modules themselves form a cdg category that we will denote by $mod^{qdg}-\mathscr{C}$. The objects are quasi modules. A degree $n$ morphisms $\alpha\in \Hom_{mod^{qdg}-\mathscr{C}}^n(F,F')$ is a natural transformation $\alpha:F\implies F'[n]$ between functors of categories enriched over graded \text{vect}or spaces. The differentials $d_{F,F'}$ are given by $d_{F,F'}(\alpha)_x:=\delta_{F'(x)}\circ \alpha_x-(-1)^n\alpha_x\circ \delta_{F(x)}$ and the curvatures are by definition $h_F(x):=h_{F(x)}-F(h_x)$.	
The full subcategory of modules inside the cdg category of quasi modules forms a dg-category which we denote by $mod^{cdg}-\mathscr{C}$. If $\mathscr{C}$ is a dg category we will write $mod^{dg}-\mathscr{C}$ for the dg category $mod^{cdg}-\mathscr{C}$. 

\subsection{Two kinds of Hochschild homology}
Let $\mathscr{C}$ be a cdg category. We define the Hochschild complex of the first kind $C_\bullet(\mathscr{C})$ as follows. As a $\Z/2$-graded vector space it is obtained by taking the direct sum along the diagonal of a bigraded vector space with one grading by $\Z$ and one by $\Z/2$. The $\Z/2$-graded vector space in degree $n$ is
\begin{align*}
	C_n(\mathscr{C})=\bigoplus_{X_0,...,X_n\in \Ob(\mathscr{C})}\Hom(X_0,X_1)\otimes\Hom(X_1,X_2)\otimes...\otimes\Hom(X_n,X_0).
\end{align*}
The differential is the sum of three differentials on $C_\bullet(\mathscr{C})$:
\begin{align*}
\begin{split}
	d_0\big(a_0[a_1|\cdots |a_k]\big)&=\sum_{i=0}^k(-1)^{|a_0|+...+|a_i|+i}a_0[\cdots|a_i|h|a_{i+1}|\cdots]\\
	d_1\big(a_0[a_1|\cdots |a_k]\big)&=\sum_{i=0}^k(-1)^{|a_0|+...+|a_{i-1}|+i}a_0[\cdots|da_i|\cdots]\\
	d_2\big(a_0[a_1|\cdots |a_k]\big)&=(-1)^{|a_0|}a_0a_1[a_2|\cdots|a_k]\\
	&+\sum_{i=1}^{k-1}(-1)^{|a_0|+...+|a_i|+i}a_0[\cdots|a_ia_{i+1}|\cdots]\\
	&+(-1)^{1+(|a_k|+1)(|a_0|+...+|a_{k-1}|+k-1)}a_ka_0[a_1|\cdots|a_{k-1}]
	\end{split}
\end{align*}
The Hochschild homology $HH_\bullet(\mathscr{C})$ is the object in $D_{Z/2}(\C)$ represented by $C_\bullet(\mathscr{C})$.
Although Hochschild homology is an important invariant of dg categories it turns out that for general cdg categories it is not very useful. For example if $h\neq 0$ in example \ref{curvedalgebra} then $HH_\bullet(A,0,h)=0$ (\cite{Cal1}, theorem 4.2). There is a notion of Hochschild homology of the second kind which in some ways is a better invariant for cdg categories. The Hochschild complex of the second kind associated to a cdg category $\mathscr{C}$, denoted $C_\bullet^{II}(\mathscr{C})$, is defined from the same bigraded vector space as for the first kind but we take direct products along the diagonal instead of direct sum. The differential is given by the same formulas as for the first kind.

There is a canonical map $C_\bullet(\mathscr{C})\to C_\bullet^{II}(\mathscr{C})$ and it's an interesting question when it is an isomorphism. 
Both Hochschild complexes of the first and second kind are functorial with respect to cdg functors as defined in this text.

\subsection{Pseudo equivalences and projective modules}
Let $\mathscr{D}$ be a cdg category. We say that an object $x\in \Ob(\mathscr{D})$ is the \textit{direct sum} of two objects $y\oplus z$ if there are degree zero morphisms $\iota_y:y\to x$ and $\iota_z:z\to x$ which induce isomorphisms of graded groups $\Hom(y,w)\times \Hom(z,w)\cong \Hom(x,w)$ and we require that $d(\iota_y)=0$ and $d(\iota_z)=0$. 

If $y\in \Ob(\mathscr{D})$ and $\tau\in \Hom^1(y,y)$ we say that an object $x\in \Ob(\mathscr{D})$ is a \textit{twist} of $y$ by $\tau$ (and we write $x=y(\tau)$) if there are degree $0$ morphisms $i:y\to x$ and $j:x\to y$ such that $ij=\id_x$, $ji=\id_y$ and $jd(i)=\tau$.

Finally we call an object $x\in \Ob(\mathscr{D})$ a \textit{degree $n$ shift} of another object $y\in \Ob(\mathscr{D})$ if there are morphisms $i:y\to x$ and $j:x\to y$ of degree $n$ and $-n$ respectively such that $ij=\id_x$, $ji=\id_y$, $d(i)=0$ and $d(j)=0$.

We call a cdg functor $F:\mathscr{C}\to \mathscr{D}$ \textit{pseudo equivalence} if it is fully faithful and any object in $\mathscr{D}$ can be obtained from objects in the image of $F$ in finitely many steps by taking direct sums, twsting, shifting or passing to direct summands. 
\begin{prop}[\cite{Pol-Pol}]\label{proponpseudoequivalences0}
	If $F:\mathscr{C}\to \mathscr{D}$ is a pseudo equivalence, then the induced map on Hochschild homology of the second kind is a quasi isomorphism. \end{prop}

For a cdg category $\mathscr{C}$ we denote by $\mathscr{C}^\#$ the underlying category enriched in graded \text{vect}or spaces. The category of graded modules on $\mathscr{C}$ (i.e. graded functors from $\mathscr{C}^\#$ to graded vector spaces) is an abelian category. The projective modules are precisely those that are direct summands of a direct sum of representable modules. We say that a module $P$ is \textit{finitely generated projective} if it is a direct summand of a finite direct sum of representable functors. Let's denote by $mod^{qdg}_{fgp}-\mathscr{C}$ the full subcategory of $mod^{qdg}-\mathscr{C}$ consisting of  quasi modules that are finitely generated projective as $\mathscr{C}^\#$-modules. Let $mod^{cdg}_{fgp}-\mathscr{C}$ denote the full subcategory of $mod^{qdg}_{fgp}-\mathscr{C}$ whose curvature vanishes. Let $\mathscr{C}=(A,0,h)$ and $\mathscr{C}'=(A,0,-h)$. We have a Yoneda functor $R:\mathscr{C}'\to mod^{qdg}_{fgp}-\mathscr{C}$ which is a cdg functor. We also have the inclusion functor $I:mod^{cdg}_{fgp}-\mathscr{C}\to mod^{qdg}_{fgp}-\mathscr{C}$.
\begin{prop}[\cite{Pol-Pol}, section 2.6]\label{proponpseudoequivalences} The functors $R$ and $I$ are pseudo equivalences and hence induce isomorphisms in Hochschild homology of the second kind.
	
\end{prop}

\subsection{Pseudo functors between dg categories}\label{pseudofunctorsbetweendgcats}
Let $\mathscr{C}$ be a dg category. The categories $Z^0(\mathscr{C})$ or $H^0(\mathscr{C})$ are obtained from $\mathscr{C}$ by taking degree zero cycles or by taking zeroth homology of all the hom-complexes in $\mathscr{C}$ respectively. Recall that the module category $mod^{dg}-\mathscr{C}$ is again a dg category and $H^0(mod^{dg}-\mathscr{C})$ has a standard triangulated structure. If we localize $H^0(mod^{dg}-\mathscr{C})$ with respect to quasi isomorphisms we get the \textit{derived category} $D(\mathscr{C})$. Using Drinfeld quotients (\cite{Drinfeld}) we get a dg enhancement of $D(\mathscr{C})$, which we denote $\mathscr{D}(\mathscr{C})$ whose objects are the same as $mod^{dg}-\mathscr{C}$. We call an object in $mod^{dg}-\mathscr{C}$ perfect if it lies in the smallest thick (meaning closed under passage to a direct summand) triangulated subcategory of $D(\mathscr{C})$ which contains the image of the Yoneda functor $H^0(\mathscr{C})\to D(\mathscr{C})$. Let $\text{perf}(\mathscr{C})\subset \mathscr{D}(\mathscr{C})$ be the full sub dg category on the perfect objects in $mod^{dg}-\mathscr{C}$. If $\mathscr{C}$ is pre-triangulated (meaning $H^0(\mathscr{C})$ is triangulated in a way compatible with the Yoneda embedding $H^0(\mathscr{C})\to H^0(mod^{dg}-\mathscr{C})$) and idempotent complete then the Yoneda functor $\mathscr{C}\to \text{perf}(\mathscr{C})$ is a quasi equivalence. If $\mathscr{C}$ and $\mathscr{C}'$ are two dg categories, both pre-triangulated then we define a \textit{pseudo functor} from $\mathscr{C}$ to $\mathscr{C}'$ to mean a dg functor $F:\mathscr{C}\to \text{perf}(\mathscr{C}')$. We will denote a pseudo functor by a dashed arrow $F:\mathscr{C}\dashrightarrow \mathscr{C}'$. By a theorem of Keller, (theorem 2.4 in \cite{Keller}), if $\mathscr{C}$ is a pretriangulated dg category then the Yoneda functor induces an isomorphism on Hochschild homology $C_\bullet(\mathscr{C})\to C_\bullet (\text{perf}(\mathscr{C}))$ and therefore any pseudo functor $F:\mathscr{C}\dashrightarrow \mathscr{C}'$ induce a well defined map in $D(k)$, $C_\bullet(\mathscr{C})\to C_\bullet(\mathscr{C}')$.

\subsection{Presheaves of cdg categories}\label{presheavesofcdg}
We will also need presheaf versions of the Hochschild complexes. A \textit{presheaf of cdg categories} $\underline{\mathscr{E}}$ on a scheme $X$ is a contravariant functor from the category of open subsets of $X$ and inclusions of such to the category of small cdg categories with cdg functors as morphisms. If for all $U$ $\underline{\mathscr{E}}(U)$ has only one object we call it a presheaf of cdg algebras and if $U\mapsto \End(\underline{\mathscr{E}}(U))$ is a sheaf we call it a sheaf of cdg algebras. If for all $U$ $\underline{\mathscr{E}}(U)$ is a dg category then we call it presheaf of dg categories
Here is an example that will be important to us later.
\begin{ex}
If $f$ is a global function on $X$ then $U\mapsto (\oo_X(U),0,f|_U)$ is a sheaf of cdg algebras which we will denote by $\oo_f$.	
\end{ex}

A \textit{morphism of presheaves of cdg categories} $\varphi:\underline{\mathscr{E}}\to \underline{\mathscr{E}}'$ is the data a cdg functor $\varphi_U:\underline{\mathscr{E}}(U)\to \underline{\mathscr{E}}'(U)$ for each open set $U\subset X$ such that whenever $V\subset U\subset X$ the following diagram commutes on the nose
$$\begin{tikzcd}
\underline{\mathscr{E}}(U)\arrow{r}{\varphi_U}\arrow{d}{res}&\underline{\mathscr{E}}'(U)\arrow{d}{res}\\
\underline{\mathscr{E}}(V)\arrow{r}{\varphi_V}&\underline{\mathscr{E}}'(V).
\end{tikzcd}
$$

If $\underline{\mathscr{E}}$ is a presheaf of cdg categories we define $\underline{C}_\bullet^{II}(\underline{\mathscr{E}})$ to be the presheaf of $\Z/2$-graded chain complexes
\begin{align*}
\underline{C}_\bullet^{II}(\underline{\mathscr{E}})(U):= C_\bullet^{II}(\underline{\mathscr{E}}(U)).
\end{align*}
Similarly one defines $\underline{C}_\bullet(\underline{\mathscr{E}})$. Any morphism $\varphi:\underline{\mathscr{C}}\to \underline{\mathscr{D}}$ of presheaves of cdg categories gives rise to a map of presheaves of complexes $\varphi_*:\underline{C}_\bullet^{?}(\underline{\mathscr{C}})\to \underline{C}_\bullet^?(\underline{\mathscr{D}})$ (where $?=I, II$). 

We will need a weaker notion of morphism between presheaves of dg categories where the dg functors are only asked to commute with the restriction functors up to compatible natural isomorphisms. Here is the exact definition.
Let $\underline{\mathscr{C}}$ and $\underline{\mathscr{D}}$ be two presheaves of dg categories on a scheme $X$.
\begin{deff}
	A lax morphism $\varphi:\underline{\mathscr{C}}\to \underline{\mathscr{D}}$ is a collection of dg functors $\{\varphi_U:\underline{\mathscr{C}}(U)\to \underline{\mathscr{D}}(U)| \ U\subset  X\text{ open}\}$ together with isomorphisms of functors $\{\alpha_{UV}:\varphi_{V}\circ \Res_{UV}\implies \Res_{UV}\circ \varphi_U| \ V\supset U\supset X \text{ open }\}$ which satisfies the following cocycle condition, for all $c\in \Ob(\underline{\mathscr{C}})$ and all $U\supset V\supset W$ we have
	$$\alpha_{UW,c}= \Res_{VW}(\alpha_{UV,c})\circ \alpha_{VW,\Res_{UV}(c)}$$
\end{deff}
Suppose $X=\cup_{i=1,...,n}U_i$ is a finite union of affine schemes. We will write $\Cech(\mathscr{U},\underline{C}_\bullet(\underline{\mathscr{C}}))$ for the total complex obtained by taking the total complex of the bicomplex $\Cech^p(\mathscr{U},\underline{C}_{-q}(\underline{\mathscr{C}}))$ (see section \ref{notation} for our conventions on the differentials in this situation). We will see that lax morphisms induce chain maps on Cech-Hochschild complexes associated to a presheaf of dg categories as above. To understand this we first introduce some notation. Let $(\varphi,\alpha)$ be a lax morphism. If $I=\{i_0<...<i_p\}$, $J=\{j_1,...,j_q\}$ and $J\cap I=\emptyset$ then we set $J_s:=\{j_1,j_2,...,j_s,i_0,...,i_p\}$ for $0\leq s\leq q$. With this notation we define the following maps $h^{q}:\Cech(\mathscr{U},\underline{C}_\bullet(\underline{\mathscr{C}}))\to \Cech(\mathscr{U},\underline{C}_\bullet(\underline{\mathscr{D}}))$ by 
\begin{align*}
C_k(\underline{\mathscr{C}}(U_{I}))\ni a_0[a_1|\cdots |a_k]\mapsto &\\
\sum_J (-1)^\epsilon\alpha\circ\varphi(a_0|_{J_q})\Big[\cdots\Big|\alpha^{-1}\Big|\varphi(a_{l_1+1}|_{J_{q-1}})|_{J_q}\Big|\cdots\Big|\alpha^{-1}\Big|\varphi(a_{l_q+1})|_{J_q}\Big|\cdots\Big]\\
\in \bigoplus_{J}C_{k+q}(\underline{\mathscr{C}}(U_{I\cup J})).
\end{align*}
\begin{align}\label{laxhomotopy}	
\end{align}
where the sum is over all ordered $J=(j_1,...,j_q)$ disjoint from $I$ and all $0\leq l_1\leq l_2\leq ...\leq l_q\leq k$. The sign is given by
$$\epsilon=(|a_0|+...+|a_{l_1}|+l_1)+...+(|a_0|+...+|a_{l_q}|+l_q)+\sigma(I,J)+pq$$
where $\sigma(I,J)$ is the sign of the permutation which arranges $(j_1,...,j_q,i_0,...,i_p)$ in order.
For example if 
$$\begin{tikzcd}
X_0&X_1\arrow{l}{a_0}&X_2\arrow{l}{a_1}&X_0\arrow{l}{a_2}	
\end{tikzcd}
$$
are three composable arrows in $\underline{\mathscr{C}}(U_I)$ and $j\notin I$ then the terms in $h^1(a_0[a_1|a_2])$ over $U_{I\cup j}$ are in bijection with paths of length 4 in the diagram
$$\begin{tikzcd}[scale cd=0.7]
&&\varphi_{U_I}(X_1)|_{U_{I\cup j}}\arrow{d}{\alpha^{-1}}&&\varphi_{U_I}(X_2)|_{U_{I\cup j}}\arrow{d}{\alpha^{-1}}\arrow{ll}{\varphi_{U_I}(a_1)|_{U_{I\cup j}}}&&\varphi_{U_I}(X_3)|_{U_{I\cup j}}\arrow{d}{\alpha^{-1}}\arrow{ll}{\varphi_{U_I}(a_2)|_{U_{I\cup j}}}\\
\varphi_{U_I}(X_0)|_{U_{I\cup j}}&&\varphi_{U_{I\cup j}}(X_1|_{U_{I\cup j}})\arrow{ll}{\alpha\circ \varphi_{U_{I\cup j}}(a_0|_{U_{I\cup j}})}&&\varphi_{U_{I\cup j}}(X_2|_{U_{I\cup j}})\arrow{ll}{\varphi_{U_{I\cup j}}(a_1|_{U_{I\cup j}})}&&\varphi_{U_{I\cup j}}(X_0|_{U_{I\cup j}})\arrow{ll}{\varphi_{U_{I\cup j}}(a_2|_{U_{I\cup j}})}
\end{tikzcd}
$$

We define $\Cech(\varphi,\alpha):\Cech(\mathscr{U},\underline{C}_\bullet(\underline{\mathscr{C}}))\to \Cech(\mathscr{U},\underline{C}_\bullet(\underline{\mathscr{C}}))$
by
$$C_k(\underline{\mathscr{C}}(U_I))\ni \bar{a}:=a_0[a_1|\cdots|a_k]\mapsto \sum_{q\geq 0}h^q(\bar{a}).$$
\begin{prop}\label{proponlaxmorphisms}
The map $\Cech(\varphi,\alpha)$ is a chain map and fits into a commutative diagram in $D_{Z/2}(\C)$

$$\begin{tikzcd}
C_\bullet(\underline{\mathscr{C}}(X))\arrow{d}{}\arrow{r}{(\varphi_X)_*}&C_\bullet(\underline{\mathscr{D}}(X))\arrow{d}{}\\
\Cech(\mathscr{U},\underline{C}_\bullet(\underline{\mathscr{C}}))\arrow{r}{\Cech(\varphi,\alpha)}&	\Cech(\mathscr{U},\underline{C}_\bullet(\underline{\mathscr{D}}))
\end{tikzcd}$$
where the vertical arrows are induced by the restriction functors.
\end{prop}
\begin{proof}
We first check that $\Cech(\varphi,\alpha)$ is a chain map which follows from the following lemma
\begin{lemma}
	(i) $d_{\Cech}\circ h^{q-1}+\bar{d}_{2}\circ h^{q}=h^{q}\circ \bar{d}_{2}+h^{q-1}\circ d_{\Cech}$ for $q\geq 1$.\\
	(ii) $\bar{d}_1\circ h^q=h^q\circ \bar{d}_1$ for all $q\geq 0$.
\end{lemma}	
\begin{proof}[Proof of lemma]
For $(i)$ we note that $\bar{d}_{2}\circ h^q(\bar{a})$ will consist of terms of the following form 
\begin{align*}&(I) \ \alpha\varphi(a_0|_{U_{I\cup J}})\alpha^{-1}\Big[\cdots \Big]\\\\
&(II) \ \alpha\varphi(a_0|_{U_{I\cup J}})\Big[\cdots \Big|\varphi((a_ra_{r+1})|_{U_{J_s}})|_{U_{I\cup J}}\Big|\cdots\Big]\\\\
&(III) \ \alpha\varphi(a_0|_{U_{I\cup J}})\Big[\cdots \Big|\varphi(a_{l_s}|_{U_{J_{q-s+1}}})|_{U_{I\cup J}}\alpha^{-1}|_{U_{I\cup J}}\Big|\cdots\Big]\\\\
&(IV) \ \alpha\varphi(a_0|_{U_{I\cup J}})\Big[\cdots \Big|\alpha^{-1}|_{U_{I\cup J}}\varphi(a_{l_{s}+1}|_{U_{J_{q-s}}})|_{U_{I\cup J}}\Big|\cdots\Big]\\\\
&(V) \ \alpha\varphi(a_0|_{U_{I\cup J}})\Big[\cdots\Big|\alpha^{-1}|_{U_{I\cup J}}\alpha^{-1}|_{U_{I\cup J}}\Big|\cdots \Big]\\\\
&(VI) \ \alpha^{-1}|_{U_{I\cup J}}\alpha\varphi(a_0|_{U_{I\cup J}})\Big[\cdots \Big]\\\\
	&(VII) \ \varphi(a_k)|_{U_{I\cup J}}\alpha\varphi(a_0|_{U_{I\cup J}})\Big[\cdots\Big]
	\end{align*}
	The terms in $(I)$ cancel with the terms in $d_{\Cech}\circ h^{q-1}(\bar{a})$, the terms in $(II)$ cancel with terms in $h^q\circ \bar{d}_2(\bar{a})$, the terms in $(III)$ cancel with the terms in $(IV)$, the terms in $(V)$ cancel with themselves, the terms in $(VI)$ cancel with the terms from $h^{q-1}\circ d_{\Cech}$ and finally the terms $(VII)$ cancel with the remaining terms in $h^q\circ \bar{d}_2(\bar{a})$.
	
	Part $(ii)$ is also just a matter of checking.
\end{proof}
Now for the second part of the proposition we note that the two composites are related by a homotopy $H$ which we now describe. On an element $\bar{a}=a_0[a_1|\cdots |a_k]$ it is defined to be
$$\sum_{q\geq 1}\tilde{h}^q(\bar{a})$$
where $\tilde{h}^q:C_k(\underline{\mathscr{C}}(X))\to \Cech^{q-1}(\mathscr{U},\underline{C}_{k+q})$ are defined exactly like $h^q$ in (\ref{laxhomotopy}) with the convention that $p=-1$. 
\end{proof}

We say that two lax-morphisms $(\varphi,\alpha),(\psi,\beta):\underline{\mathscr{C}}\to \underline{\mathscr{D}}$ are isomorphic if there are natural isomorphisms $\tau_U:\varphi_U\overset{\sim}{\implies} \psi_U$ such that for any $V\subset U$ we have a commutative diagram of natural isomorphisms
$$\begin{tikzcd}
	\Res^U_V\circ \varphi_U\arrow{r}{\Res^U_V(\tau_U)}&\Res^U_V\circ \psi_U\\
	\varphi_V\circ \Res^U_V\arrow{u}{\alpha^U_V}\arrow{r}{\tau_{V}}&\psi_V\circ\Res^U_V\arrow{u}{\beta^U_V}.
\end{tikzcd}$$ 
\begin{prop}
	If $(\varphi,\alpha)$ and $(\psi,\beta)$ are isomorphic lax morphisms of presheaves of dg categories $\underline{\mathscr{C}}\to \underline{\mathscr{D}}$ then the induced maps are homotopic $\Cech(\varphi,\alpha)\simeq\Cech(\psi,\beta)$.
\end{prop}
\begin{proof}
	Let $\tau:(\varphi,\alpha)\to (\psi,\beta)$ be an isomorphism. The following is a chain homotopy between $\Cech(\varphi,\alpha)$ and $\Cech(\psi,\beta)$:
	\begin{align*}
	&\sum_{q\geq 0}H^q,\text{ where }H^q(a_0[a_1|\cdots|a_k])	:=\\
	\sum\sum &(-1)^{\epsilon'}\tau|_{U_{I\cup J}}\alpha\varphi(a_0|_{U_{I\cup J}})\Big[\cdots\Big|\alpha^{-1}\Big|\varphi(a_{l_i+1}|_{U_{J_{q-i}}})|_{U_{I\cup J}}\Big|\cdots\\
&\Big|\tau|_{U_{I\cup J}}^{-1}\Big|\cdots\Big|\beta^{-1}|_{U_{I\cup J}}\Big|\psi(a_{l_{i+1}+1})|_{U_{I\cup J}}\Big|\cdots\Big]
	\end{align*}
	where the outer sum is indexed as in the definition of $h^q$ (see equation $(\ref{laxhomotopy})$ above) and the inner sum runs from $0$ to $k+q$ picking out the position for $\tau^{-1}$. The sign is 
	$$\epsilon'=\epsilon+|a_0|+...+|a_r|+r+s$$
	where $\epsilon$ is as in the definition of $h^q$, $a_r$ is the last morphism among $a_0,...,a_k$ that appear before $\tau^{-1}$ and $s$ is the number of $\alpha^{-1}$'s that appear before $\tau^{-1}$. 
\end{proof}

If $\underline{\mathscr{C}}$ is a presheaf of dg algebras (recall that this means $\underline{\mathscr{C}}(U)$ only has one object for all $U\subset X$), then any two lax morphisms of the form $(\varphi,\alpha),(\varphi,\beta):\underline{\mathscr{C}}\to \underline{\mathscr{D}}$ are isomorphic. Indeed, lets call the object in $\underline{\mathscr{C}}(X)$ $c$. Then we can define $\tau_{U,c|_U}:=(\beta^X_U)^{-1}\circ\alpha^X_U$. Now if $\varphi$ is an actual morphism of presheaves of dg categories then $(\varphi,\id)$ is a lax morphism and if $\underline{\mathscr{C}}$ is a presheaf of dg algebras then any other lax morphism $(\varphi,\alpha)$ is isomorphic to $(\varphi,\id)$. Note that a morphism of presheaves of dg categories $\varphi$ induces a morphism $$\Cech(\varphi):\Cech(\mathscr{U},\underline{C}_\bullet(\underline{\mathscr{C}}))\to \Cech(\mathscr{U},\underline{C}_\bullet(\underline{\mathscr{D}}))$$
by taking $(\varphi_{U_I})_*$ on the summand $C_\bullet(\underline{\mathscr{C}}(U_I))$. 
\begin{lemma}\label{comparinglaxandstrict}
If $\varphi:\underline{\mathscr{C}}\to \underline{\mathscr{D}}$ is a morphism of presheaves of dg categories, then the maps $\Cech(\varphi)$ and $\Cech(\varphi,\id)$ are chain homotopic.
\end{lemma}
\begin{proof}
	Note that with $\alpha$ being the identity natural transformation all the $h^q$'s in the definition of $\Cech(\varphi,\id)$, with $q>1$, vanish. Therefore $\Cech(\varphi,\id)-\Cech(\varphi)=h^1$. Let $\bar{a}=a_0[a_1|\cdots|a_k]\in C_k(\underline{\mathscr{C}}(U_I))$, where $I=\{i_0<...<i_p\}$ and define
	$$H(\bar{a})=\sum_{j,l_1,l_2}(-1)^{\tau}\varphi (a_0)[\cdots|\varphi(a_{l_1})|\id|\cdots|\varphi(a_{l_2})|\id|\cdots]\in \bigoplus C_{k+2}(\underline{\mathscr{D}}(U_{I\cup j}))$$
	where the sum is over all $j\notin I$, and $0\leq l_1\leq l_2\leq k$ and the sign is given by
	$$\tau=(|a_0|+...+|a_{l_1}|+l_1)+(|a_0|+...+|a_{l_2}|+l_2)+\sigma(I,j).$$
	Then $(\bar{d}_1+\bar{d}_2+d_{\Cech})H+H(\bar{d}_1+\bar{d}_2+d_{\Cech})=-h^1=\Cech(\varphi)-\Cech(\varphi,\id)$.
\end{proof}

\section{Hochschild homology of category of matrix factorizations}
In this section we introduce matrix factorizations and recall the Hochschild homology for some categories of matrix factorizations.
\subsection{Affine matrix factorizations}
Let $A$ be a commutative ring and $h\in A$ and consider the cdg algebra $A_h=(A,0,h)$ from example \ref{curvedalgebra}. Then a matrix factorization of $h$ is nothing but an object in $A_h-mod^{cdg}_{fgp}$. More precisely, such an object consists of the following data
$$\begin{tikzcd}
	P^0\arrow[shift left]{r}{\delta^0}&P^1\arrow[shift left]{l}{\delta^1}
\end{tikzcd}$$
where $P^0,P^1$ are projective $A$-modules, $\delta^0,\delta^1$ are $A$-module homomorphisms such that the composites $\delta^1\delta^0$ and $\delta^0\delta^1$ are both multiplication by $h$.

\subsection{Global matrix factorizations}\label{globalfact}
Now let $X$ be a smooth quasi compact scheme and $f\in \oo_X(X)$ a global function. We will define the derived category of coherent matrix factorizations and a dg enhancement of this. First, consider the dg category $\text{fact}(X,f)$ whose objects consist of the following data

$$\begin{tikzcd}
	\mathcal{F}^\bullet=\mathcal{F}^0\arrow[shift left]{r}{\delta^0}&\mathcal{F}^1\arrow[shift left]{l}{\delta^1}
\end{tikzcd}$$
where $\mathcal{F}^0,\mathcal{F}^1$ are coherent $\oo_X$-modules and $\delta^0,\delta^1$ are $\oo_X$-module homomorphisms such that $\delta^0\delta^1 $ and $\delta^1\delta^0$ both equal multiplication by $f$. The hom spaces are defined as 
\begin{align*}
\Hom_{\text{fact}(X,f)}(\mathcal{F}^\bullet,\mathcal{G}^\bullet):=\Hom_{\oo_X}^\bullet(\mathcal{F}^\bullet,\mathcal{G}^\bullet)	
\end{align*}
with differential $\partial(\alpha):=\delta_{\mathcal{G}^\bullet}\alpha-(-1)^{|\alpha|}\alpha\delta_{\mathcal{F^\bullet}}$. The category $Z^0(\text{fact}(X,f))$ is abelian and $H^0(\text{fact}(X,f))$ is triangulated. We define $D(\text{fact}(X,f))$ as the Verdier quotient of $H^0(\text{fact}(X,f))$ by the full thick triangulated subcategory generated by all totalizations of short exact sequences in $Z^0(\text{fact}(X,f))$. 

Next we define $\text{vect}(X,f)$ to be the full dg subcategory of $\text{fact}(X,f)$ on the objects whose underlying sheaves $\mathcal{F}^0,\mathcal{F}^1$ are \text{vect}or bundles. If $X=\Spec(A)$ is affine then this is the category $mod^{cdg}_{fgp}-A_f$ defined earlier and if $X$ is smooth  and affine then $H^0(\text{vect}(X,f))\simeq D(\text{fact}(X,f))$. Under some extra assumptions on $f$ this gives an important example where the two kinds of Hochschild homology agree.
\begin{prop}[\cite{Pol-Pol}, corollary A section 4.7]\label{hoch1and2forvect}
	If $X$ is a smooth affine variety and $f|_{X\setminus V(f)}$ is smooth then the natural map $C_\bullet(\text{vect}(X,f))\to C_\bullet^{II}(\text{vect}(X,f))$ is a quasi isomorphism.
\end{prop}

If $X$ not affine the category $\text{vect}(X,f)$ will not work as a dg enhancement for $D(\text{fact}(X,f))$. Instead we have to introduce the dg category $\text{vect}_{\Cech}(X,f)$. Let $\mathscr{U}$ denote a finite affine cover of $X$. Then we define $\text{vect}_{\Cech}(X,f)$ to have objects same as $\text{vect}(X,f)$ but for the hom complexes we take
\begin{align}\label{cechhoms}
	\Hom_{\text{vect}_{\Cech}(X,f)}(\mathcal{P}^\bullet,\mathcal{Q}^\bullet):=Tot(\Cech{^\bullet}(\mathscr{U},\underline{\Hom}^\bullet(\mathcal{P}^\bullet,\mathcal{Q}^\bullet))).
\end{align}
This is a dg enhancement of $D(\text{fact}(X,f))$ when $X$ is smooth.

We will also need the cdg category $\text{qfact}(X,f)$. Its objects are defined exactly as the objects of $\text{fact}(X,f)$ except we don't require anything about the composites $\delta^0\delta^1$ or $\delta^1\delta^0$. The hom-spaces and the differentials are defined in the same way as for $\text{fact}(X,f)$ and the curvature of an object $\mathcal{F}^{\bullet}$ is the endomorphism $\delta^2-\rho_f$. We define $\text{qvect}(X,f)$ to be the full sub cdg category of $\text{qfact}$ whose underlying sheaves are \text{vect}or bundles. We define the cdg category $\text{qvect}_{\Cech}(X,f)$ by altering the hom complexes as in (\ref{cechhoms}).

Finally we will also need presheaf versions of some of these. We define the following presheaf of dg categories
\begin{align*}
\underline{\text{vect}}_{\Cech}(X,f)(U):=\text{vect}_{\Cech}(U,f|_U).
\end{align*}
The presheaves of cdg categories $\underline{\text{qvect}}_{\Cech}(X,f)$, $\underline{\text{vect}}(X,f)$ and $\underline{\text{qvect}}(X,f)$ are defined similarly.
\begin{rmk}
Note that a lot of the dg and cdg categories introduced in this section are not small. In order to still make sense of the standard complexes computing their Hochschild homology we therefore have to restrict to some small subcategories containing at least one object from each isomorphism class. In what follows we will do this implicitly and will assume all categories to be  small.
\end{rmk}

\subsection{Hochschild homology of $\text{vect}_{\Cech}(X,f)$}\label{secthochofmat}
The Hochschild homology of $\text{vect}_{\Cech}(X,f)$ will be denoted $HH_\bullet(X,f)$. For a smooth variety $X$ where $f|_{X\setminus V(f)}$ is smooth this was computed in \cite{Efimov}, it is derived global sections of the complex of sheaves $(\Omega^\bullet_X,-df\wedge(-))$. Below is a slight adaption of the computation in \cite{Efimov}. 

There is a sequence of maps which we will show are all quasi isomorphisms.

$$\begin{tikzcd}[scale cd=0.85]	
C_\bullet(\text{vect}_{\Cech} (X,f))\arrow{r}{(1)}&\Cech(\mathscr{U},\underline{C}_\bullet(\underline{\text{vect}}_{\Cech}(X,f)))&\Cech(\mathscr{U},\underline{C}_\bullet(\underline{\text{vect}}(X,f)))\arrow{l}{(2)}\arrow{dll}{(3)}\\
\Cech(\mathscr{U},\underline{C}^{II}_\bullet(\underline{\text{vect}}(X,f)))\arrow{r}{(4)}&\Cech(\mathscr{U},\underline{C}^{II}_\bullet(\underline{\text{qvect}}(X,f)))&\Cech(\mathscr{U},\underline{C}^{II}_\bullet(\oo_{-f}))\arrow{l}{(5)}\arrow{dll}{(6)}\\
\Cech(\mathscr{U},(\Omega_X^\bullet,-df\wedge-)).
\end{tikzcd}$$
\begin{align}\label{totalhkrf}	
\end{align}

We will first discuss $(6)$ (See \cite{Cal1}, theorem 4.2 b).

\begin{prop}\label{hkrxf}
The following defines a quasi isomorphism
\begin{align*}\begin{split}
	&HKR_{(X,f)}:\Cech(\mathscr{U},\underline{C}_\bullet^{II}(\oo_f))\to \Cech(\mathscr{U},(\Omega_X^{\bullet},df\wedge-)),\\
	C_\bullet^{II}(\oo_f(U_{i_0\cdots i_p}))&\ni a_0[a_1|\cdots |a_k]\mapsto \frac{1}{k!}a_0da_1\wedge...\wedge da_k\in \Omega^k_X(V\cap U_{i_0\cdots i_p})
\end{split}\end{align*}
	
\end{prop}
To prove this we will need two lemmas
\begin{lemma}
	Suppose that we have an unbounded exact chain complex of inverse systems of abelian groups
	\begin{align*}
		S_\bullet:=(...\to S_n\to S_{n-1}\to S_{n-2}\to...).
	\end{align*}
and that each term is Mittag-Leffler. Then the complex
	\begin{align*}
		...\to\lim S_n\to \lim S_{n-1}\to\lim S_{n-2}\to...
	\end{align*}
	is exact.
\end{lemma}
\begin{proof}[Proof of lemma]
Split the long exact sequence into short exact sequences
$$\begin{tikzcd}[scale cd=0.5]
	&&0\arrow{dr}{}&&0&&\\
	&&&B_n=Z_{n-1}\arrow{ru}{}\arrow{dr}{}&&&\\
	...\arrow{rr}{}\arrow{dr}{}&&S_n\arrow{rr}{}\arrow{ru}{}&&S_{n-1}\arrow{dr}{}\arrow{rr}{}&&S_{n-2}\\
	&B_{n+1}=Z_n\arrow{ru}{}\arrow{dr}{}&&&&B_{n-1}=Z_{n-2}\arrow{ru}{}\arrow{dr}{}&\\
	0\arrow{ru}{}&&0&&0\arrow{ru}{}&&0
\end{tikzcd}$$	
Now the inverse systems $B_k=Z_{k-1}$ are quotients of Mittag-Leffler systems and are therefore themselves Mittag-Leffler. It follows that after applying $\underset{\leftarrow}{\lim}$ termwise the diagonal short exact sequences remain exact and then it follows that the middle long exact sequence remains exact.
\end{proof}

\begin{lemma}
	Let $A^{\bullet,\bullet}$ and $B^{\bullet,\bullet}$ be two bicomplexes whose vertical and horizontal differentials both decrease the degrees. For each $n$ let $A^{\leq n,\bullet}\subset A^{\bullet,\bullet}$ denote the subcomplex truncated horizontally at $n$. Let $A^{\geq n,\bullet}$ be the quotient of $A^{\bullet,\bullet}$ by $A^{\leq n-1,\bullet}$. Suppose further that for each $n$, the diagonals of $A^{\geq n,\bullet}$ and $B^{\geq n,\bullet}$ only have finitely many non zero terms. If $\alpha:A^{\bullet,\bullet}\to B^{\bullet,\bullet}$ is a morphism of bicomplexes which is a quasi isomorphism with respect to the vertical  differentials, then $Tot^{\prod}(\alpha):Tot^{\prod}(A^{\bullet,\bullet})\to Tot^{\prod}(B^{\bullet,\bullet})$ is a quasi isomorphism.
\end{lemma}

\begin{proof}
	Consider the inverse systems of chain complexes
	\begin{align*}S(A)&:=\{...\to\text{Tot}^{\oplus}A^{\geq -2,\bullet}\to \text{Tot}^{\oplus} A^{\geq -1,\bullet}\to \text{Tot}^{\oplus}A^{\geq 0,\bullet}\}\\
	S(B)&:=\{...\to \text{Tot}^{\oplus}B^{\geq -2,\bullet}\to \text{Tot}^{\oplus} B^{\geq -1,\bullet}\to \text{Tot}^{\oplus}B^{\geq 0,\bullet}\}
	\end{align*}
where the maps are induced by canonical quotient maps of bicomplexes. The map $\alpha$ induces a morphism $S(\alpha):S(A)\to S(B)$ of complexes of inverse systems. The assumption on $\alpha$  and the Mapping Theorem for filtered modules (\cite{Macl},theorem XI.3.4) imply that $S(\alpha)$ is a termwise quasi isomorphism, in other words a quasi isomorphism of complexes of inverse systems. But then $Cone(S(\alpha))$ is an acyclic complex of inverse systems of abelian groups, all of whose terms are Mittag-Leffler and therefore, by the previous lemma,
\begin{align*}
	Cone(\underset{\leftarrow}{\lim}S(\alpha))=\underset{\leftarrow}{\lim}Cone(S(\alpha))=0.
\end{align*}
We see that $\underset{\leftarrow}{\lim}S(\alpha):\text{Tot}^{\prod}A^{\bullet,\bullet}\to \text{Tot}^{\prod}B^{\bullet,\bullet}$ is a quasi isomorphism.
\end{proof}

\begin{proof}[Proof of proposition]
	The affine case follows from the classical HKR theorem and by applying the previous lemma to the bicomplexes
	$$A^{\bullet,\bullet}:=\begin{tikzcd}
		 &\vdots\arrow{d}{}&\vdots\arrow{d}{} &\vdots\arrow{d}{} \\
	\dots	&\oo^{\otimes 3} \arrow{l}{}\arrow{d}{d_2}&\oo^{\otimes 2}\arrow{l}{d_0} \arrow{d}{d_2}&\oo \arrow{l}{d_0}\\
	\dots	&\oo^{\otimes 2} \arrow{l}{}\arrow{d}{d_2}&\oo \arrow{l}{d_0}&\\
	\dots	&\oo \arrow{l}{}&&\\
		&&&
	\end{tikzcd}$$
	and	
	$$B^{\bullet,\bullet}:=\begin{tikzcd}
		 &\vdots\arrow{d}{}&\vdots\arrow{d}{} &\vdots\arrow{d}{} \\
	\dots	&\Omega^2 \arrow{l}{}\arrow{d}{0}&\Omega\arrow{l}{df\wedge } \arrow{d}{0}&\oo \arrow{l}{df\wedge}\\
	\dots	&\Omega\arrow{l}{}\arrow{d}{0}&\oo \arrow{l}{df\wedge}&\\
	\dots	&\oo \arrow{l}{}&&\\
		&&&
	\end{tikzcd}$$
	
For the non-affine case, think of $\text{Tot}^{\prod}A^{\bullet,\bullet}$ and $\text{Tot}^{\prod}B^{\bullet,\bullet}$ as presheaves of complexes. We want to show that \begin{align*}
	\begin{split}
		{\Cech}{^p}(\mathscr{U},(\text{Tot}^{\prod}A^{\bullet,\bullet})_q)\to{\Cech}{^p}(\mathscr{U},(\text{Tot}^{\prod}B^{\bullet,\bullet})_q)
	\end{split}
\end{align*}
is a quasi isomorphism. By considering total complexes and filtering this in the $p$-direction we obtain a map of filtered complexes. It then follows from the Mapping Theorem of filtered complexes and the affine case already considered that this map of bicomplexes gives rise to a quasi isomorphism of total complexes. This proves the proposition.\end{proof}

Now $(5)$ is a quasi isomorphism because for any $U_I=U_{i_0\cdots i_p}$ the map $C^{II}_\bullet(\oo_{U_I,-f|_{U_I}})\to C^{II}_\bullet(\text{qvect}(U_I,f|_{U_I}))$ is a quasi isomorphism by proposition \ref{proponpseudoequivalences}. The map $(4)$ is a quasi isomorphism for exactly the same reasons. The map $(3)$ is a quasi isomorphism because each  $C_\bullet(\text{vect}(U_I,f|_{U_I}))\to C^{II}_\bullet(\text{vect}(U_I,f|_{U_I}))$ is a quasi isomorphism by proposition \ref{hoch1and2forvect}. The map $(2)$ is a quasi isomorphism because for each $U_I$ we have a quasi equivalence $\text{vect}(U_I,f|_{U_I})\to \text{vect}_{\Cech}(U_I,f|_{U_I})$ which induce quasi isomorphisms on standard complexes. The proof that the map $(1)$ is a quasi isomorphism is similar to the proof of proposition 5.1 in \cite{Efimov}. We include the proof here with some adjustments to make it fit our situation.
\begin{prop}
	Let $\mathscr{C}$ be one of the two presheaves of dg categories $\underline{\text{perf}}_{\Cech}(X)$ or $\underline{\text{vect}}_{\Cech}(X,f)$. The natural map $C_\bullet(\mathscr{C}(X))\to \Cech(\mathscr{U},\underline{C}_\bullet(\mathscr{C}))$ is a quasi isomorphism.
\end{prop}
\begin{proof}
We prove it for $\mathscr{C}=\underline{\text{vect}}_{\Cech}(X,f)$. Recall that $\mathscr{U}$ is referring to a fixed affine open covering $X=U_1\cup...\cup U_n$. Let $U=U_n$ and $V=U_1\cup...\cup U_{n-1}$. Let $Z=X\setminus U$. For any open $W\subset X$ let $\mathscr{C}(W)_{Z\cap W}\subset \mathscr{C}(W)$ denote the full sub dg category on the objects which become acyclic in $\mathscr{C}(W\cap U).$ Let $\mathscr{V}$ denote the open covering  $V=U_1\cup U_3\cup...\cup U_{n-1}$. Let $\mathscr{C}'$ denote the presheaf of dg categories on $V$
$$V\supset W\mapsto \text{vect}_{\Cech(\mathscr{V})}(W,f|_W).$$
We have a natural functor $\mathscr{R}es: \mathscr{C}(X)\to \mathscr{C}'(V)$ which restricts objects on $X$ to $V$ and which is defined on homomorphism complexes by
$$\Cech(\mathscr{U},\underline{\Hom}_X^\bullet(\mathcal{F}^\bullet,\mathcal{G}^\bullet))\to \Cech(\mathscr{U}\cap V,\underline{\Hom}_V^\bullet(\mathcal{F}|_V^\bullet,\mathcal{G}|_V^\bullet))\twoheadrightarrow \Cech(\mathscr{V},\underline{\Hom}_V^\bullet(\mathcal{F}|_V^\bullet,\mathcal{G}|_V^\bullet)).$$
After passing to homotopy categories this functor fits into a commutative diagram
$$\begin{tikzcd}
H^0(\mathscr{C}(X))\arrow{d}{\sim}\arrow{r}{H^0(\mathscr{R}es)}&H^0(\mathscr{C}'(V))\arrow{d}{\sim}\\	
D(\text{fact}(X,f))\arrow{r}{Res}&D(\text{fact}(V,f|_V))
\end{tikzcd}.
$$

We have the following diagram in $D_{Z/2}(\C)$
$$\begin{tikzcd}
C_\bullet(\mathscr{C}(X)_Z)\arrow{r}{}\arrow{d}{}&C_\bullet(\mathscr{C}(X))\arrow{r}{}\arrow{d}{(\mathscr{R}es)_*}&C_\bullet(\mathscr{C}(U))\arrow{r}{}\arrow{d}{}&C_\bullet(\mathscr{C}(X)_Z)[1]\arrow{d}{}\\
C_\bullet(\mathscr{C}'(V)_Z)\arrow{r}{}&C_\bullet(\mathscr{C}'(V))\arrow{r}{}&C_\bullet(\mathscr{C}'(V\cap U))\arrow{r}{}&C_\bullet(\mathscr{C}'(V)_Z[1])	
\end{tikzcd}
$$
where the rows are distinguished. (See \cite{Efimov}, proposition 5.1 and its proof.)
Since $Z\subset V$ the left most and the right most vertical arrows are quasi isomorphisms. Therefore we get a Mayer-Viertoris exact triangle
$$\begin{tikzcd}
	C_\bullet(\mathscr{C}(X))\arrow{r}{}&C_\bullet(\mathscr{C}(U))\oplus C_\bullet(\mathscr{C}'(V))\arrow{r}{}&C_\bullet(\mathscr{C}'(V\cap U))\arrow{dll}{}\\
	C_\bullet(\mathscr{C}(X))[1]\arrow{r}{}&...&
\end{tikzcd}
$$
We now claim that there is an exact triangle of the form
$$\begin{tikzcd}
	\Cech(\mathscr{U},\underline{C}_\bullet(\mathscr{C}))\arrow{r}{\alpha}&\Cech(\mathscr{U}\cap U,\underline{C}_\bullet(\mathscr{C}))\oplus \Cech(\mathscr{V},\underline{C}_\bullet(\mathscr{C}'))\arrow{r}{\beta}&\Cech(\mathscr{V}\cap U,\underline{C}_\bullet(\mathscr{C}'))\arrow{dll}{}\\
	\Cech(\mathscr{U},\underline{C}_\bullet(\mathscr{C}))[1]\arrow{r}{}&...&
\end{tikzcd}
$$
to which the first exact triangle naturally maps. Assuming this claim for now the rest of the proof goes as follows. Because the open covering $\mathscr{U}\cap U$ of $U$ has $U$ itself as one of its open sets the map $C_\bullet(\mathscr{C}(U))\to \Cech(\mathscr{U}\cap U,\underline{C}_\bullet(\mathscr{C}))$ is a quasi isomorphism. By induction on the number of open sets the maps $C_\bullet(\mathscr{C}'(V))\to \Cech(\mathscr{V},\underline{C}_\bullet(\mathscr{C}'))$ and $C_\bullet(\mathscr{C}'(V\cap U))\to \Cech(\mathscr{V}\cap U,\underline{C}_\bullet(\mathscr{C}'))$ are quasi isomorphisms. The proposition then follows by the five lemma. 

To construct the second distinguished triangle, we first define what the two maps $\alpha$ and $\beta$ are. First, let $A\subset \Cech(\mathscr{U},\underline{C}_\bullet(\mathscr{C}))$ be the subcomplex consisting of all summands $C_\bullet(\mathscr{C}(U_I))$ with $n\in I$ and $B=C_\bullet(\mathscr{C}(U_I))/A$ so that as a vector space we have $C_\bullet(\mathscr{C}(U_I))=A\oplus B$. Similarly let $E\subset \Cech(\mathscr{U}\cap U,\underline{C}_\bullet(\mathscr{C}))$ be the sub complex consisting of the summands $C_\bullet(\mathscr{C}(U_I\cap U_n))$ with $n\in I$ and $F=\Cech(\mathscr{U}\cap U,\underline{C}_\bullet(\mathscr{C}))/E$ and $G=\Cech(\mathscr{V},\underline{C}_\bullet(\mathscr{C}'))$ so that as vector spaces we have $\Cech(\mathscr{U}\cap U,\underline{C}_\bullet(\mathscr{C}))\oplus \Cech(\mathscr{V},\underline{C}_\bullet(\mathscr{C}'))=E\oplus F\oplus G$. Finally let $H=\Cech(\mathscr{V}\cap U,\underline{C}_\bullet(\mathscr{C}'))$. In this notation, the maps $\alpha$ and $\beta$ can be described by matrices
$$\alpha=\begin{bmatrix}\id & \ 0 & \ 0\\0& \ \alpha_0 & \alpha_1 \end{bmatrix}:A\oplus B\to E\oplus F\oplus G$$
$$\beta=\begin{bmatrix}0\\ \beta_0\\\beta_1\end{bmatrix}:E\oplus F\oplus G\to H$$
where $\alpha_0$ is made up by the restrictions $C_\bullet(\mathscr{C}(U_I))\to C_\bullet(\mathscr{C}(U_I\cap U))$ ($n\notin I$), $\alpha_1$ is a quasi isomorphism which is induced by the quasi equivalences $\mathscr{C}(U_I)\to \mathscr{C}'(U_I)$ ($n\notin I$). The map $\beta_0$ is also a quasi isomorphism, it's induced by the quasi equivalences $\mathscr{C}(U_I\cap U)\to \mathscr{C}'(U_I\cap U)$ ($n\notin I$) and $\beta_1$ is $-1$ times the map induced by the restrictions $\mathscr{C}'(U_I)\to \mathscr{C}'(U_I\cap U)$ ($n\notin I$). In this situation one can show that the canonical maps $Cone(\alpha)\to H$ is a quasi isomorphism and therefore we obtain the required distinguished triangle. 
\end{proof}

\subsection{The case $f=0$}\label{thecasef0}
When $f=0$ we have two intresting categories. 

On the one hand the $\Z$-graded dg-category $\text{perf}_{\Cech}(X)$ whose objects are finite $\Z$-graded complexes of vector bundles and whose morphism complexes are defined as
$$\Hom_{\text{perf}_{\Cech}(X)}(\mathcal{P}^\bullet,\mathcal{Q}^\bullet):=Tot(\Cech^\bullet(\mathscr{U},\underline{\Hom}_X^\bullet(\mathcal{P}^\bullet,\mathcal{Q}^\bullet))).$$
We can then think of $\text{perf}_{\Cech}(X)$ as a $\Z/2$-graded category by forgetting the $\Z$-grading on the  homomorphism complexes and only remembering the parity of each degree.

On the other hand we can consider the $\Z/2$-graded category $\text{vect}_{\Cech}(X,0)$ defined in section \ref{globalfact}.
Note that we have a fully faithful dg functor functor $\text{perf}_{\Cech}(X)\to \text{vect}_{\Cech}(X,0)$ which forgets the $\Z$-grading on the objects and only remembers the parity of each term of a complex. In fact, we have the following proposition.
\begin{prop}\label{proponperfandvect}
If $X$ is a nonsingular variety, then any object in $\text{vect}_{\Cech}(X,0)$ is homotopy equivalent to an object in $\text{perf}_{\Cech}(X)$ and hence they are quasi equivalent.
\end{prop}
\begin{proof}
Let $\begin{tikzcd}
		\mathcal{F}^{\bullet}:=\Big[\overset{deg 1}{\mathcal{F}^0}\arrow[shift left]{r}{\delta^0}&\overset{deg 0}{\mathcal{F}^1}\arrow[shift left]{l}{\delta^1} \ \Big]
\end{tikzcd}$. Let $\mathcal{K}^i=\ker(\delta^i)$ and $\mathcal{Q}^i=\im(\delta^i)$. Then the following is an exact triangle in $D(\text{fact}(X,0))$
$$\begin{tikzcd}\mathcal{K}^0\arrow[shift left]{d}{0}\arrow{r}{}&\mathcal{F}^0\arrow[shift left]{d}{\delta^0}\arrow{r}{}&\mathcal{Q}^0\arrow[shift left]{d}{0}\arrow{r}{}&\mathcal{K}^1\arrow[shift left]{d}{0}\\
\mathcal{K}^1\arrow[shift left]{u}{0}\arrow{r}{}&\mathcal{F}^1\arrow{r}{}\arrow[shift left]{u}{\delta^1}&\mathcal{Q}^1\arrow[shift left]{u}{0}\arrow{r}{}&\mathcal{K}^0\arrow[shift left]{u}{0}
\end{tikzcd}$$
Both $(\mathcal{K}^\bullet,0)$ and $(\mathcal{Q}^\bullet,0)$ are isomorphic to objects in $H^0(\text{perf}_{\Cech}(X))$. The reason is because $X$ is non-singular the components of $\mathcal{K}^\bullet$ and $\mathcal{Q}^\bullet$ have finite resolutions by locally free sheaves. It follows that $\mathcal{F}^\bullet$ too is isomorphic to an object in $H^0(\text{perf}_{\Cech}(X))$.
\end{proof}
Now let's consider the sequence

$$\begin{tikzcd}[scale cd=0.85]	
C_\bullet(\text{perf}_{\Cech} (X))\arrow{r}{}&\Cech(\mathscr{U},\underline{C}_\bullet(\underline{\text{perf}}_{\Cech}(X)))&\Cech(\mathscr{U},\underline{C}_\bullet(\underline{\text{perf}}(X)))\arrow{l}{}\arrow{dll}{=}\\
\Cech(\mathscr{U},\underline{C}_\bullet(\underline{\text{perf}}(X)))\arrow{r}{=}&\Cech(\mathscr{U},\underline{C}_\bullet(\underline{\text{perf}}(X)))&\Cech(\mathscr{U},\underline{C}_\bullet(\oo_X))\arrow{l}{}\arrow{dll}{}\\
\Cech(\mathscr{U},\Omega_X^\bullet)
\end{tikzcd}$$
From each term in the above sequence there is a natural map to the corresponding term of (\ref{totalhkrf}) forming commutative squares, the map $\Cech(\mathscr{U},\Omega_X^\bullet)\to \Cech(\mathscr{U},\Omega_X^\bullet)$ being the identity. All of the maps above are quasi isomorphisms and this forces all of the maps from this sequences to (\ref{totalhkrf}) to be quasi isomorphisms

\section{Main problem}
In this section we discuss the main problem.

\subsection{Setup}
Let $X$ be a smooth complex variety and let us fix a finite affine covering $\mathscr{U}$ of $X$. Let $f\in H^0(X,\oo_X)$ and suppose $Y\subset X$ is a smooth divisor such that $f|_Y=0$. We assume also that $f|_{X\setminus f^{-1}(0)}$ is smooth. The categories we are mostly interested in are $D^b(\text{coh}Y)$ (usual derived category) and $D(\text{fact}(X,f))$.

We have dg enhancements for both $D^b(\text{coh}(Y))$ and $D(\text{fact}(X,f))$. For $D^b(\text{Y})$ we take the category $\text{\text{perf}}_{\Cech}(Y)$ and for $D(\text{fact}(X,f))$ we take $\text{vect}_{\Cech}(X,f)$.
We have a functor $i_*:D^b(\text{coh}Y)\to D(\text{fact}(X,f))$. It lifts to a quasi functor $\underline{i}_*:\text{\text{perf}}_{\Cech}(Y)\dashrightarrow \text{vect}_{\Cech}(X,f)$ by which we mean that the following diagram commutes
$$\begin{tikzcd}
	&H^0(\text{\text{perf}}(\text{vect}_{\Cech}(X,f)))&\\
	H^0(\text{\text{perf}}_{\Cech}(Y))\arrow{ru}{H^0(\underline{i}_*)}\arrow{d}{\sim}&&H^0(\text{vect}_{\Cech}(X,f))\arrow{dl}{\sim}\arrow{ul}{Yoneda}\\
	D^b(\text{coh}Y)\arrow{r}{i_*}&D(\text{fact}(X,f))\arrow{uu}{Yoneda}&.
\end{tikzcd}$$
The quasi functor $\underline{i}_*$ induces a map in Hochschild homology. We will compute this map in terms of the identifications $$HH_\bullet(Y):=HH_\bullet(\text{perf}_{\Cech}(Y))\simeq \Cech(\mathscr{U}\cap Y,\Omega_Y^\bullet)$$ and $$HH_\bullet(X,f):=HH_\bullet(\text{vect}_{\Cech}(X,f))\simeq \Cech(\mathscr{U},(\Omega^\bullet_X,-df\wedge(-)))$$ from sections \ref{thecasef0} and \ref{secthochofmat} respectively.

\subsection{Replacing $\underline{i}_*$ by a dg functor}\label{P}	
Consider the following matrix factorization $$\begin{tikzcd}
		P:=\Big[\overset{deg 1}{\oo_X(-Y)}\arrow[shift left]{r}{\cdot x}&\overset{deg 0}{\oo_X}\arrow[shift left]{l}{\cdot g} \ \Big]
\end{tikzcd}$$
where $x\in H^0(X,\oo_X(Y))$ is an equation for $Y$ and $g\in H^0(X,\oo_X(-Y))$ is such that $gx=f\in H^0(X,\oo_X)$. In $D(\text{fact}(X,f))$ we have $i_*\oo_Y\cong P$. Indeed, the kernel of the canonical map $P\to i_*\oo_Y$ is the matrix factorization $$\begin{tikzcd}\Big[\overset{deg 1}{\oo_X(-Y)}\arrow[shift left]{r}{\cdot 1}&\overset{deg 0}{\oo_X(-Y)}\arrow[shift left]{l}{\cdot f} \ \Big]
\end{tikzcd}$$
which is zero in the derived category because it maps to zero under the equivalence $\text{Cok}:D(\text{fact}(X,f))\to D^b_{sing}(V(f))$ (See \cite{Orlov-Cok} for details about this equivalence).

Consider also the sheaf of dg algebras
$\cala:=[\oo_X(-Y)\overset{x}{\to} \oo_X].$ It is quasi isomorphic to $\oo_Y$ as sheaves of dg algebras on $X$. The quasi isomorphism $\cala\to \oo_Y$ also induces a quasi isomorphism on Hochschild complexes for any affine open $V\subset X$, $C_\bullet(\cala|_V)\overset{\sim}{\to}C_\bullet(\oo_Y|_V)$.

We define the dg category $\text{perf}_{\Cech}(\cala)$ to have objects, same as $\text{perf}(X)$. The homomorphism complexes are given by
\begin{align*}
	\Hom_{\text{perf}_{\Cech}(\cala)}(\mathcal{V},\mathcal{W}):=\Cech(\mathscr{U},\underline{\Hom}_\cala^\bullet(\mathcal{V}\otimes \cala,\mathcal{W}\otimes \cala)).
\end{align*}
We now introduce a dg functor $\text{Ind}:\text{\text{perf}}_{\Cech}(\cala)\to \text{\text{perf}}_{\Cech}(Y)$. On objects it is defined by 
$$\text{Ind}(\mathcal{V})=\mathcal{V}\otimes i_*\oo_Y.$$
The chain maps on homomorphism complexes are induced by the following morphism of complexes of sheaves
$$\begin{tikzcd}\underline{\Hom}^\bullet_{\cala}(\mathcal{V}\otimes \cala,\mathcal{W}\otimes \cala)\arrow{r}{-\otimes_\cala i_*\oo_Y}	&\underline{\Hom}^\bullet_{\oo_Y}(\mathcal{V}\otimes \cala\otimes_{\cala}i_*\oo_Y,\mathcal{W}\otimes \cala\otimes_{\cala}i_*\oo_Y)\arrow{dl}{\cong}\\
\underline{\Hom}^\bullet_{\oo_Y}(\mathcal{V}\otimes i_*\oo_Y,\mathcal{W}\otimes i_*\oo_Y).&
\end{tikzcd}
$$
\begin{lemma}\label{indfunctor}
The functor Ind defined above induces an isomorphism in $D_{Z/2}(\C)$
$$Ind_*:C_\bullet(\text{perf}_{\Cech}(\cala))\to C_\bullet (perf_{\Cech}(Y)).$$
\end{lemma}
\begin{proof}
For a dg category $\mathscr{C}$, let $\tilde{\mathscr{C}}$ denote its pretriangulated hull. We get a strictly commutative diagram of dg functors
$$\begin{tikzcd}\widetilde{\text{perf}_{\Cech}(\cala)}\arrow{r}{\widetilde{\text{Ind}}}&\widetilde{\text{perf}_{\Cech}(Y)}\\
\text{perf}_{\Cech}(\cala)\arrow{u}{}\arrow{r}{Ind}&\text{perf}_{\Cech}(Y)\arrow{u}{}\end{tikzcd}$$
where the vertical functors induce isomorphism in Hochschild homology (\cite{Keller}) and the right vertical arrow is a quasi equivalence since $\text{perf}_{\Cech}(Y)$ is pretriangulated. Therefore it is enough to check that $\widetilde{\text{Ind}}$ induces an isomorphism in Hochschild homology. We show that $\widetilde{\text{Ind}}$ is in fact a quasi equivalence.
To see that $\widetilde{\text{Ind}}$ induces a quasi isomorphism on homomorphism complexes we note that $\text{Ind}$ applied to hom-complexes can be factored as a composite of quasi isomorphisms: 
	\begin{align*}
	\begin{split}
		\Hom_{\text{\text{perf}}_{\Cech}(\cala)}(\mathcal{V}^\bullet\otimes \cala,\mathcal{W}^\bullet\otimes \cala)&= \Cech(\mathscr{U},\underline{\Hom}^\bullet_\cala(\mathcal{V}^\bullet\otimes \cala,\mathcal{W}^\bullet\otimes \cala))\\
		&\simeq \Cech(\mathscr{U},\underline{\Hom}^\bullet_{\oo_X}(\mathcal{V}^\bullet,\mathcal{W}^\bullet\otimes \cala))\\
		&\simeq \Cech(\mathscr{U},(\mathcal{V}^\vee)^\bullet\otimes \mathcal{W}^\bullet\otimes \cala)\\
		&\overset{\sim}{\to} \Cech(\mathscr{U},(\mathcal{V}^\vee)^\bullet\otimes \mathcal{W}^\bullet\otimes \oo_Y)\\
		& \simeq \Cech(\mathscr{U},\underline{\Hom}^\bullet_{\oo_X}(\mathcal{V}^\bullet,\mathcal{W}^\bullet\otimes \oo_Y))\\
		&\simeq \Cech(\mathscr{U},\underline{\Hom}^\bullet_{\oo_Y}(\mathcal{V}^\bullet\otimes \oo_Y,\mathcal{W}^\bullet\otimes \oo_Y))\\
		&\simeq \Hom_{\text{\text{perf}}_{\Cech}(Y)}(\text{Ind}(\mathcal{V}^\bullet\otimes \cala),\text{Ind}(\mathcal{W}^\bullet\otimes \cala))
		\end{split}
	\end{align*}
	This means that $H^0(\text{Ind})$ is fully faithful and therefore so is $H^0(\widetilde{\text{Ind}})$. Therefore, to show that $H^0(\widetilde{\text{Ind}})$ is is essentially surjective it is enough to check that the smallest triangulated subcategory of $H^0(\widetilde{\text{perf}_{\Cech}(Y)})$ containing the image of $H^0(\widetilde{\text{Ind}})$ is all of $H^0(\widetilde{\text{perf}_{\Cech}(Y)})$. 
	 Let $\oo_X(1)$ denote a very ample line bundle on $X$ and let $\oo_Y(1)=\oo_X(1)|_Y$. We know that $\{\oo_Y(n)\}_{n\in \Z}$ classically generates $H^0(\text{perf}_{\Cech}(Y))\simeq \widetilde{H^0(\text{perf}_{\Cech}(Y)})$. Essential surjectivity then follows from the observation that
	 $$\oo_Y(n)=\text{Ind}( \oo_X(n)).$$
\end{proof}

Let $\cala^\#$ denote the underlying sheaf of graded algebras of $\cala$ and $P^\#$ denote the underlying sheaf of $\Z/2$-graded $\oo_X$-modules of $P$. As graded $\oo_X$-modules we have $\cala^\#=P^\# $ and this gives us a multiplication $\mu:\cala^\#\otimes P^\#\to P^\#$. If $\mathcal{M}$ is an $\cala$-module we can thus make sense of the tensor product $\mathcal{M}\otimes_\cala P$ as a graded $\oo_X$-module. In fact, the morphism
\begin{align*}
	\mathcal{M}\otimes \cala\otimes P\to \mathcal{M}\otimes P, \ m\otimes a\otimes p\mapsto ma\otimes p-m\otimes ap
\end{align*}
whose cokernel by definition is the tensor product $\mathcal{M}\otimes_\cala P$ is a morphism of matrix factorizations of $f$. Therefore the product $\mathcal{M}\otimes_\cala P$ is naturally a matrix factorization of $f$. We get a dg functor 
\begin{align}\label{tensorwp}
	-\otimes_\cala P:\text{\text{perf}}_{\Cech}(\cala)\to \text{vect}_{\Cech}(X,f)
\end{align}
which on objects is defined by $\mathcal{V}\mapsto \mathcal{V}\otimes_X P$ and the chain maps on hom complexes is induced by the morphism of complexes of sheaves
$$\begin{tikzcd}
\underline{\Hom}^\bullet_{\cala}(\mathcal{V}\otimes \cala,\mathcal{W}\otimes \cala)\arrow{r}{-\otimes_\cala P}	&\underline{\Hom}^\bullet_{\oo_Y}(\mathcal{V}\otimes \cala\otimes_{\cala}P,\mathcal{W}\otimes \cala\otimes_{\cala}P)\arrow{dl}{\cong}\\
\underline{\Hom}^\bullet_{\oo_Y}(\mathcal{V}\otimes P,\mathcal{W}\otimes P).&
\end{tikzcd}
$$
\begin{lemma}\label{resolutionforoy}
	The following diagram commutes up to natural quasi isomorphism
	
	$$\begin{tikzcd}
		\text{\text{perf}}_{\Cech}(\cala)\arrow{d}{Ind}\arrow{r}{\otimes_\cala P}&\text{vect}_{\Cech}(X,f)\arrow{d}{Yoneda}\\
		\text{\text{perf}}_{\Cech}(Y)\arrow{r}{\underline{i}_*}&\text{\text{perf}}(\text{vect}_{\Cech}(X,f))
	\end{tikzcd}$$
	
\end{lemma}
\begin{proof}
	For any $Q\in \text{vect}_{\Cech}(X,f)$, the natural map $$\Hom_{\text{vect}_{\Cech}(X,f)}(Q,\mathcal{V}\otimes P)\to \Hom_{\text{fact}_{\Cech}(X,f) }(Q,\mathcal{V}\otimes \oo_Y)$$ is a quasi isomorphism where $\text{fact}_{\Cech}(X,f)$ is the DG category with objects same as $\text{fact}(X,f)$ and whose homomorphism complexes are defined in the same way as in $\text{vect}_{\Cech}(X,f)$.
\end{proof}
In fact, we can define similar functors
\begin{align*}\text{Ind}_U&:\text{perf}_{\Cech}(\cala|_U)\to \text{perf}_{\Cech}(Y\cap U),\\
-\otimes_{\cala|_U}P|_U&:\text{perf}_{\Cech}(\cala|_U)\to \text{vect}_{\Cech}(U,f|_U)
\end{align*}
for any open set $U$ and assemble them into lax morphisms of presheaves of dg categories.
\begin{prop}\label{laxind}
There are natural isomorphisms $$\alpha_{U,V,\mathcal{F}}: \text{Ind}_V(\mathcal{F}|_V)\implies \text{Ind}_U(\mathcal{F})|_V$$
and
$$\beta_{U,V,\mathcal{F}}:\mathcal{F}|_V\otimes_{\cala|_V}P|_V\implies (\mathcal{F}\otimes_{\cala|_U}P|_U)|_V$$
making $$(\text{Ind},\alpha):\underline{\text{perf}}_{\Cech}(\cala)\to \underline{\text{perf}}_{\Cech}(Y)$$ 
and 
$$(-\otimes_\cala P,\beta):\underline{\text{perf}}_{\Cech}(\cala)\to \underline{\text{vect}}_{\Cech}(X,f)$$
into lax morphisms of presheaves of dg categories.
\end{prop}
\begin{proof}
For $Ind$ this is just the statement that for $W\subset U\subset V$ we have a commutative diagram of natural isomorphisms
$$\begin{tikzcd}
	\mathcal{V}|_W\otimes_W \oo_{Y\cap W}&(\mathcal{V}|_V\otimes_V \oo_{Y\cap V})|_W\arrow{l}{\cong}\\
	&(\mathcal{V}\otimes_X \oo_Y)|_{W}\arrow{lu}{\cong}\arrow{u}{\cong}.
\end{tikzcd}
$$	
It may be easier to see that the following diagram commutes
$$\begin{tikzcd}
	(j_W)_*(\mathcal{V}|_W\otimes_W \oo_{Y\cap W})&(j_V)_*(\mathcal{V}|_V\otimes_V \oo_{Y\cap V})\arrow{l}{}\\
	&\mathcal{V}\otimes_X \oo_Y\arrow{lu}{}\arrow{u}{}
\end{tikzcd}
$$	
where $j_W:W\to U$ and $j_V:V\to U$ are the inclusions.

The case of $\otimes_\cala P$ is similar.
\end{proof}

\subsection{HKR type map for $\mathcal{A}$}
Let $\Omega^k_X(\log Y)$ denote the algebraic $k$-forms on $X$ with logarithmic poles along $Y$. It is a subsheaf of $\Omega_X^k(Y)$ whose sections locally look like $\omega+\frac{dx}{x}\wedge\omega'$ where $\omega\in \Omega_X^k$, $\omega'\in \Omega^{k-1}_X$ and $x$ is a local equation for $Y$. There is a short exact sequence of complexes 
$$
\begin{tikzcd}	
	0\arrow{r}{}&(\Omega_X^\bullet,-df\wedge)[1]\arrow{r}{L}&(\Omega_X^\bullet(\log Y),-df\wedge)[1]\arrow{r}{}&(\Omega_Y^\bullet,0)\arrow{r}{}&0
\end{tikzcd}
$$
\begin{align}\label{SES}
\end{align}
where the first map is the inclusion and the second map is locally defined by $\beta+\frac{dx}{x}\wedge \alpha \mapsto \alpha|_Y$.
It gives rise to a quasi isomorphism 
\begin{align}
	\begin{split}\label{qis2}
		\Omega_\cala^{\bullet}:=\text{Cone}(L)\overset{\sim}{\to}\Omega^\bullet_Y.
	\end{split}
\end{align}

We will construct a commutative diagram
$$\begin{tikzcd}
	{\Cech}(\mathscr{U},\underline{C}_\bullet(\oo_Y))\arrow{d}{HKR_Y}&{\Cech}(\mathscr{U},\underline{C}_\bullet(\cala))\arrow{d}{HKR_\cala}\arrow{l}{}\\
	{\Cech}(\mathscr{U},\Omega_Y^\bullet)&{\Cech}(\mathscr{U},\Omega_{\cala}^\bullet)\arrow{l}{}.
\end{tikzcd}$$
\begin{align}\label{hkrdiagram}
\end{align}
The horizontal arrows are induced by the quasi isomorphisms $\cala\to \oo_Y$  and (\ref{qis2}) respectively and the left vertical arrow comes from the classical HKR theorem. It remains to define the map $HKR_\cala$. Let $x_i\in \oo_X(U_i)$ be local equations for $Y\subset X$. Let $u_{ij}=x_jx_i^{-1}$ on the intersections $U_{ij}$. Then we identify $\cala|_{U_j}\cong [\oo_{U_i}\epsilon_i\to\oo_{U_i}]$ where $\epsilon_i$ is in degree $-1$ and the differential maps $\epsilon_i$ to $x_i$. Note that on the intersections we have $u_{ij}\epsilon_i=\epsilon_j$. Let $U_I=U_{i_0\cdots i_p}$ and let $\bar{a}:=a_0[a_1|\cdots |a_k]\in {\cala^0(U_I)}^{\otimes k+1}\subset C_k(\cala(U_I))$. Then define
\begin{align*}
\begin{split}
	HKR_\cala(\bar{a})&:=\frac{1}{k!}a_0 \frac{dx_{i_0}}{x_{i_0}}\wedge da_1\wedge...\wedge da_k\\
	&+\frac{1}{k!}a_0dx_{i_0}\wedge dg_{i_0}\wedge da_1\wedge...\wedge da_k\\
	&-\sum_{j<i_0}\frac{(-1)^p}{k!}a_0\frac{du_{i_0j}}{u_{i_0j}}\wedge da_1\wedge...\wedge da_k\\
	&\in \Gamma\Big(U_I,\Omega^{k+1}_X(\log Y))\oplus \Gamma(U_I,\Omega_X^{k+2})\oplus  (\bigoplus_{j<i_0}\Gamma(U_{j\cup I},\Omega_X^{k+1})\Big)
	\end{split}
\end{align*}
and
\begin{align*}
	\begin{split}
		HKR_\cala(a_0[a_1|\cdots|a_l\epsilon_{i_0}|\cdots|a_k]):= \frac{(-1)^{l}}{k!}a_0dx_{i_0}\wedge da_1\wedge ...\wedge da_k.
	\end{split}
\end{align*}
and $HKR_\cala$ vanishes on elements which have two or more factors from $\cala^{-1}$.
\begin{prop}\label{proponhkra}The following is true about the map $HKR_\cala$\\

	1) $HKR_\cala$ is a chain map.\\
	
	2) The diagram (\ref{hkrdiagram}) commutes.\\
	
	3) $HKR_\cala$ is a quasi isomorphism.
\end{prop}
\begin{proof}
	We check that $HKR_\cala$ is a chain map at the level of presheaves. We do this first on elements of internal degree zero. As before let $U_I=U_{i_0\cdots i_p}$ and let $\bar{a}:=a_0[a_1|\cdots |a_k]\in \cala^0(U_I)^{\otimes k+1}\subset C_k(\cala(U_I))$.
	
\begin{align}
	\begin{split}\label{formula1}
		HKR_\cala(\bar{d}_2\bar{a})&=0\\
		HKR_\cala(\bar{d}_1\bar{a})&=0\\
		HKR_\cala(d_{Cech}\bar{a})&=HKR_\cala\Big(\sum_{j\notin I}(-1)^{Sgn(j,I)}\bar{a}|_{U_{I\cup j}}\Big)\\
		&=HKR_\cala\Big(\sum_{j<i_0}\bar{a}|_{U_{I\cup j}}+\sum_{j\notin I, \ j>i_0}(-1)^{Sgn(j,I)}\bar{a}|_{U_{I\cup j}}\Big)\\
		&=\frac{1}{k!}\sum_{j<i_0}a_0\frac{dx_{j}}{x{j}}\wedge da_1\wedge...\wedge da_k|_{j\cup I} \\
		&+\frac{1}{k!}\sum_{j<i_0}a_0dx_{j}\wedge dg_{j}\wedge da_1\wedge...\wedge da_k|_{j\cup I}\\
		&-\frac{1}{k!}\sum_{j'<j<i_0}(-1)^{p+1}a_0\frac{du_{jj'}}{u_{jj'}}\wedge da_1\wedge ...\wedge da_k|_{j'\cup j\cup I}\\
		&+HKR_\cala \Big(\sum_{j\notin I, \ j>i_0}(-1)^{Sgn(j,I)}\bar{a}|_{U_{I\cup j}}\Big).
	\end{split}
\end{align}
	On the other hand we have
	\begin{align}
		\begin{split}\label{formula2}
			d_{Cech}HKR_\cala(\bar{a})&=\frac{1}{k!}\sum_{j<i_0}a_0 \frac{dx_{i_0}}{x_{i_0}}\wedge da_1\wedge...\wedge da_k|_{j\cup I}\\
	&+\frac{1}{k!}\sum_{j<i_0}a_0 dx_{i_0}\wedge dg_{i_0}\wedge da_1\wedge...\wedge da_k|_{j\cup I}\\
	&-\frac{1}{k!}\sum_{j<j'<i_0}(-1)^pa_0\Big(\frac{du_{i_0j'}}{u_{i_0j'}}-\frac{du_{i_0j}}{u_{i_0j}}\Big)\wedge da_1\wedge...\wedge da_k|_{j\cup j'\cup I}\\
	&+\sum_{j\notin I,j>i_0}HKR_\cala(\bar{a})|_{{j\cup I}}
		\end{split}
	\end{align}
and
\begin{align}
		\begin{split}\label{formula3}
			\bar{d}_{Cone}HKR_\cala(\bar{a})&=(-1)^{p}a_0df\wedge \frac{dx_{i_0}}{x_{i_0}}\wedge da_1\wedge...\wedge da_k\\
	&+(-1)^{p}L(a_0 dx_{i_0}\wedge dg_{i_0}\wedge da_1\wedge...\wedge da_k)\\
	&+\sum_{j<i_0}a_0df\wedge \frac{du_{i_0j}}{u_{i_0j}}\wedge da_1\wedge...\wedge da_k|_{j\cup I}\\
	&+\sum_{j<i_0}L(a_0 \frac{du_{i_0j}}{u_{i_0j}}\wedge da_1\wedge...\wedge da_k)|_{j\cup I}
		\end{split}
	\end{align}
From here one can check that $(\bar{d}_{Cone}+d_{Cech})HKR_\cala(\bar{a})=HKR_\cala(d_{Cech}+\bar{d}_1+\bar{d}_2)(\bar{a})$. Indeed the first sum in (\ref{formula1}) cancels with with the first sum in (\ref{formula2}) and the last sum in (\ref{formula3}). The second sum in (\ref{formula1}) cancels with the second sum in (\ref{formula2}) and the sum on the third row of (\ref{formula3}). The third sum in (\ref{formula1}) cancels with the third sum in (\ref{formula2}). The last sum in (\ref{formula1}) cancels with the last sum in (\ref{formula2}). Finally the first two terms in (\ref{formula3}) cancel with each other.

Next, let's check that $(\bar{d}_{Cone}+d_{Cech})HKR_\cala(\bar{b})=HKR_\cala(d_{Cech}+\bar{d}_1+\bar{d}_2)(\bar{b})$ where $\bar{b}=b_0[b_1|\cdots|b_l\epsilon_{i_0}|\cdots |b_k]$ has internal degree $-1$. We have
\begin{align*}
	\begin{split}
		HKR_\cala(\bar{d}_2\bar{b})&=0\\
		HKR_\cala(\bar{d}_1\bar{b})&=(-1)^{l+p}\frac{1}{k!}b_0\frac{dx_{i_0}}{x_{i_0}}\wedge db_1\wedge...\wedge d(b_lx_{i_0})\wedge...\wedge db_k\\
		&+(-1)^{l+p}\frac{1}{k!}b_0dx_{i_0}\wedge dg_{i_0}\wedge db_1\wedge...\wedge d(b_lx_{i_0})\wedge...\wedge db_k\\
		&-\sum_{j<i_0}(-1)^{l}\frac{1}{k!} b_0\frac{du_{i_0j}}{u_{i_0j}}\wedge db_1\wedge...\wedge d(b_lx_{i_0})\wedge...\wedge db_k\\
		HKR_\cala(d_{Cech}\bar{b})&=\sum_{j\notin I, \ j>i_0}(-1)^{l+Sgn(j,I)}\frac{1}{k!}b_0dx_{i_0}\wedge db_1\wedge...\wedge db_l\wedge...\wedge db_k\\
		&+\sum_{j<i_0}(-1)^{l}\frac{1}{k!}b_0dx_{j}\wedge db_1\wedge...\wedge d(u_{ji_0}b_l)\wedge...\wedge db_k
	\end{split}
\end{align*}
On the other hand we have
\begin{align*}
	\begin{split}
		d_{Cech}HKR_\cala(\bar{a})&=\sum_{j\notin I}(-1)^{Sgn(j,I)+l}\frac{1}{k!}b_0dx_{i_0}\wedge db_1\wedge...\wedge db_l\wedge...\wedge db_k\\
		\bar{d}_{Cone}HKR_\cala(\bar{b})&=(-1)^{l+p+1}\frac{1}{k!}b_0df\wedge dx_{i_0}\wedge db_1\wedge...\wedge db_l\wedge...\wedge db_k\\
		&+(-1)^{l+p}\frac{1}{k!}L(b_0dx_{i_0}\wedge db_1\wedge...\wedge db_l\wedge...\wedge db_k)
	\end{split}
\end{align*}
Again the relevant equality can be checked from these formulas. This proves the first part of the proposition. The last thing one has to check for the first part is that  $HKR_\cala\circ\bar{d}_1$ vanishes on elements with two factors from $\cala^{-1}$.

The second part can be checked from the definitions. The last part then follows from the second because all the other arrows in the diagram are quasi isomorphisms.
\end{proof}

\subsection{Multiplying by the Todd class}\label{inversetodd}
We have the wedge product map $\Omega^\bullet_X\otimes \Omega_X^\bullet\to \Omega_X^\bullet$ which we can use to make $\Cech(\mathscr{U},\Omega_X^\bullet)$ into a dg algebra by defining
\begin{align*}
	(\alpha\wedge \beta)_{i_0i_1\cdots i_{p+q}}:=(\alpha_{i_0\cdots i_p}\wedge\beta_{i_p\cdots i_{p+q}})|_{U_{i_0\cdots i_{p+q}}}
\end{align*}
where $\alpha\in \Cech{^p}(\mathscr{U},\Omega^\bullet_X)$ and $\beta\in \Cech{^q}(\mathscr{U},\Omega_X^\bullet)$.
The same formula makes \\$\Cech(\mathscr{U},\Omega_X^\bullet(\log Y))$ into a left $\Cech(\mathscr{U},\Omega_X^\bullet)$ module. We will need a slight variation of this. Define
\begin{align*}
	(\alpha\overline{\wedge}\beta)_{i_0i_1\cdots i_{p+q}}:=(-1)^{pq}(\beta\wedge\alpha)_{i_0i_1\cdots i_{p+q}}.
\end{align*}
Then $\Cech(\mathscr{U},\Omega^\bullet_X)$ is a dg algebra with the multiplication $\bar{\wedge}$ as well. Similarly we define two dg algebra structures on $\Cech(\mathscr{U},\Omega^\bullet_Y)$. Moreover, the natural map $\Cech(\mathscr{U},\Omega_X^\bullet)\to \Cech(\mathscr{U},\Omega_Y^\bullet)$ is a dg algebra homomorphism with respect to either $\wedge$ or $\overline{\wedge}$. Now, we define a right $\Cech(\mathscr{U},\Omega_X^\bullet)$-module structure on $\Cech(\mathscr{U},\Omega_X^\bullet\oplus \Omega_X^\bullet(\log Y)[1])$ by 
\begin{align}\label{modulestructureonomegal}
(\alpha+\beta)	\bar{\wedge}\gamma:=(-1)^{pq}\gamma\wedge \alpha+(-1)^{(p+1)q}\gamma\wedge \beta
\end{align}
for $\alpha\in \Cech{^p}(\mathscr{U},\Omega_X^\bullet)$, $\beta\in \Cech{^p}(\mathscr{U},\Omega_X^\bullet(\log Y)[1])$ and $\gamma\in \Cech{^q}(\mathscr{U},\Omega_X^\bullet).$

The Cech cocycle $(u_{ij}^{-1})_{j<i}\in \Cech{^1}(\mathscr{U},\oo_X^\times)$ corresponds to the line bundle $\oo_X(-Y)$. The Chern class $c_1(-Y)$ is then the Cech cocycle $(u_{ij}{du^{-1}_{ij}})_{j<i}\in\Cech{^1}(\mathscr{U},\Omega^1_X)$. The Todd class of $\oo_X(-Y)$ is defined as the formal power series $\frac{x}{e^x-x}$ evaluated at $c_1(-Y)$. It's inverse is
\begin{align*}
	\text{Td}(-Y)^{-1}:=\sum_{q\geq 0}\frac{c_1(-Y)^{\overline{\wedge}q}}{(q+1)!}=\sum_{q\geq 0}(-1)^{\binom{q}{2}}\frac{c_1(-Y)^{\wedge q}}{(q+1)!}.
\end{align*}
\begin{prop}\label{propontodd}
The map 
\begin{align*}
	(-)\overline{\wedge}\text{Td}(-Y)^{-1}:\Cech(\mathscr{U},\Omega_\cala^\bullet)\to \Cech(\mathscr{U},\Omega^\bullet_\cala)
\end{align*}
commutes with the differentials and gives rise to a commutative diagram
$$\begin{tikzcd}
	\Cech(\mathscr{U},\Omega_Y^\bullet)\arrow{d}{(-)\overline{\wedge}Td(-Y)^{-1}}&\Cech(\mathscr{U},\Omega_\cala^\bullet)\arrow{l}{}\arrow{d}{(-)\overline{\wedge}Td(-Y)^{-1}}\\
	\Cech(\mathscr{U},\Omega_Y^\bullet)&\Cech(\mathscr{U},\Omega_\cala^\bullet)\arrow{l}{}
\end{tikzcd}$$
\end{prop}
\begin{proof}
	First we check that $(-)\bar{\wedge}Td(-Y)^{-1}$ commutes with the differential on $\Cech(\Omega_\cala^\bullet)$. Let $\alpha\in \Cech{^p}(\mathscr{U},\Omega_X^k)\subset \Cech(\mathscr{U},\Omega_\cala^\bullet)$. Then 
	\begin{align*}
	\begin{split}
		d(\alpha)\bar{\wedge}c_1(-Y)^{\bar{\wedge}q}=&((-1)^pL(\alpha)+d_{Cech}(\alpha)-(-1)^pdf\wedge\alpha)\bar{\wedge}c_1(-Y)^{\bar{\wedge}q}\\
		=&(-1)^{p+(p+1)q}c_1(-Y)^{\bar{\wedge}q}\wedge L(\alpha)\\
		&+(-1)^{(p+1)q}c_1(-Y)^{\bar{\wedge} q}\wedge d_{Cech}(\alpha)\\
		&-(-1)^{p+pq}c_1(-Y)^{\bar{\wedge}q}\wedge df\wedge\alpha\\
		=&(-1)^{p+(p+1)q}L(c_1(-Y)^{\bar{\wedge}q}\wedge \alpha)\\
		&+(-1)^{pq}d_{Cech}(c_1(-Y)^{\bar{\wedge} q}\wedge \alpha)\\
&-(-1)^{p+q+pq}df\wedge c_1(-Y)^{\bar{\wedge}q}\wedge \alpha\\
=&d(\alpha\bar{\wedge}c_1(-Y)^{\bar{\wedge}q})
		\end{split}
	\end{align*}
	Similarly if $\beta\in \Cech{^p}(\mathscr{U},\Omega^k_X(\log Y))\subset \Cech(\mathscr{U},\Omega^\bullet_\cala)$ we check that
	\begin{align*}
		\begin{split}
			d(\beta)\bar{\wedge}c_1(-Y)^{\bar{\wedge}q}=&(-1)^pdf\wedge\beta\bar{\wedge}c_1(-Y)^{\bar{\wedge}q}+d_{Cech}(\beta)\bar{\wedge}c_1(-Y)^{\bar{\wedge}q}\\
			=&(-1)^{p+(p+1)q}c_1(-Y)^{\bar{\wedge}q}\wedge df\wedge\beta\\
			&+(-1)^{(p+2)q}c_1(-Y)^{\bar{\wedge}q}\wedge d_{Cech}(\beta)\\
			=&(-1)^{p+pq}df\wedge c_1(-Y)^{\bar{\wedge}q}\wedge \beta\\
			+&(-1)^{(p+2)q-q}d_{Cech}(c_1(-Y)^{\bar{\wedge}q}\wedge \beta)\\
			=&d(\beta\bar{\wedge}c_1(-Y)^{\bar{\wedge}q}).
		\end{split}
	\end{align*}
Commutativity of the diagram can be checked immediately from the definitions.
\end{proof}

\subsection{Trace map for a small cdg category of quasi modules}\label{constructingphi}
Let $\{P_1,...,P_r\}$ be a collection of objects in $\text{qvect}(X,f)$ and let $\mathscr{C}$ be the presheaf of cdg categories defined by $U\mapsto \langle P_1|_U,...,P_r|_U\rangle$
where the latter refers to the full sub cdg category of $\text{qvect}(U,f|_U)$. Assume that we have trivialisations $P_i|_{U_j}\cong \oo_{U_j}^{n_i}\oplus \oo_{U_j}^{m_i}$. Our goal is to define a map 
\begin{align}\begin{split}\label{goalisphi}
\phi:\Cech(\mathscr{U},\underline{C}^{II}_\bullet(\mathscr{C}))\to \Cech(\mathscr{U},\underline{C}^{II}_\bullet(\oo_{-f}))\end{split}\end{align}
such that if $(P_i,\delta_{P_i})=(\oo_X,0)$ for some $1\leq i\leq r$ then $\phi$ is a quasi inverse to the quasi isomorphism
\begin{align}\label{yon}\Cech(\mathscr{U},\underline{C}_\bullet^{II}(\oo_{-f}))\to \Cech(\mathscr{U},\underline{C}^{II}_\bullet(\mathscr{C}))\end{align} induced by the Yoneda functor.\\

Given a morphism $\alpha:P_i(U_l)\to P_j(U_l)$ we will denote by $(\alpha)_l$ the matrix defined by 
$$\oo_{U_l}^{n_i}\oplus \oo_{U_l}^{m_i}\cong P_i\rightarrow P_j\cong \oo_{U_l}^{n_j}\oplus \oo_{U_l}^{m_j}.$$
Let $g_{kl}$ denote the change of basis matrix such that 
\begin{align*}\begin{split}g_{kl}(\alpha)_lg_{kl}^{-1}=(\alpha)_k\end{split}\end{align*}
on $U_k\cap U_l$. 

We will also need the preasheaf of cdg categories $\mathscr{F}$ which to an open set $U$ assigns the subcategory of $\text{qvect}(U,f|_U)$ on all objects of the form 
$$\begin{tikzcd}
		\Big[\oo_U^{\oplus n}\arrow[shift left]{r}{0}&\oo_U^{\oplus m}\arrow[shift left]{l}{0} \ \Big].
\end{tikzcd}$$
We will define a collection of maps
\begin{align*}\begin{split}h^q:\Cech(\mathscr{U},\underline{C}_\bullet^{II}(\mathscr{C}))\to \Cech(\mathscr{U},\underline{C}_\bullet^{II}(\mathscr{F}))	.
\end{split}\end{align*}
Let
$\bar{\alpha}:=\alpha_0[\alpha_1|\cdots |\alpha_k]\in C_\bullet^{II}(\mathscr{C})(U_{i_0\cdots i_p})$ and define
\begin{align*}\begin{split}h^q(\bar{\alpha}):=\\
\sum(-1)^{\epsilon+pq+\binom{q}{2}}g_{i_0i_{-q}}(\alpha_0)_{i_{-q}}\Big[(\alpha_1)_{i_{-q}}\Big|\cdots\Big|(\alpha_{l_1})_{i_{-q}}\Big|g_{i_{1-q}i_{-q}}^{-1}\Big|(\alpha_{l_1+1})_{i_{1-q}}\Big|\cdots \\
\cdots\Big|(\alpha_{l_2})_{i_{1-q}}\Big|g_{i_{2-q}i_{1-q}}^{-1}\Big|(\alpha_{l_2+1})_{i_{2-q}}\Big|\cdots \Big|(\alpha_{l_q})_{i_{-1}}\Big|g_{i_{0}i_{-1}}^{-1}\Big|(\alpha_{l_q+1})_{i_{0}}\Big|\cdots\Big]	
\end{split}\end{align*}
where the sum is over all $q$-tuples $1\leq i_{-q}<i_{1-q}<i_{2-q}<...<i_{-1}<i_{0}$ and all $q$-tuples $0\leq l_1\leq...\leq l_q\leq k$. The sign is given by 
\begin{align*}\begin{split}\epsilon=\sum_{i=1}^q\Big(\sum_{j=0}^{l_i}|\alpha_j|+l_i\Big)
\end{split}\end{align*}

\begin{lemma}\label{lemmaonhq}
We have
\begin{align*}\begin{split}
	\bar{d}_2h^q+d_{Cech} h^{q-1}=h^{q-1}d_{Cech}+h^q\bar{d_2}.
\end{split}\end{align*}	
\end{lemma}

\begin{proof}
	We apply both sides to $\bar{\alpha}=\alpha_0[\alpha_1|\cdots |\alpha_k]$ which is a section over $U_{i_0...i_p}$. First, $\bar{d}_2h^q(\bar{\alpha})$ will consist of terms of the form
	\begin{align}
		\begin{split}\label{typ1}
			&\color{black}(-1)^{p+q+|\alpha_0|+\epsilon+pq+\binom{q}{2}}g_{i_0i_{1-q}}(\alpha_0)_{i_{1-q}}\Big[(\alpha_1)_{i_{1-q}}\Big|\cdots \Big],\\
			&\color{black}(-1)^{p+q+|\alpha_0|+...+|\alpha_{l_j}|+l_j+j+\epsilon+pq+\binom{q}{2}}g_{i_0i_{-q}}(\alpha_0)_{i_{-q}}\Big[\cdots\Big|g_{i_{j+1-q}i_{j-1-q}}^{-1}\Big|\cdots\Big],\\
			&\color{black}(-1)^{p+q+|\alpha_0|+...+|\alpha_{r}|+r+j+\epsilon+pq+\binom{q}{2}}g_{i_0i_{-q}}(\alpha_0)_{i_{-q}}\Big[\cdots \Big|(\alpha_{r}\alpha_{r+1})_{j-q}\Big|\cdots\Big],\\
			&\color{black}(-1)^{p+q+j+1+|\alpha_0|+|\alpha_1|+...+|\alpha_{l_{j}}|+l_j+\epsilon+pq+\binom{q}{2}}g_{i_0i_{-q}}(\alpha_0)_{i_{-q}}\Big[\cdots \Big|g^{-1}_{i_{j-q}i_{j-1-q}}(\alpha_{l_{j}+1})_{i_{j-q}}\Big|\cdots\Big]\\
			&\color{black}(-1)^{p+q+j+|\alpha_0|+...+|\alpha_{l_{j}+1}|+l_{j}+1+\epsilon+pq+\binom{q}{2}}g_{i_0i_{-q}}(\alpha_0)_{i_{-q}}\Big[\cdots \Big|(\alpha_{l_{j}+1})_{i_{j-1-q}}g^{-1}_{i_{j-q}i_{j-1-q}}\Big|\cdots\Big]\\
			&\color{black}(-1)^{p+q+1+|\alpha_0|+...+|\alpha_k|+k+q-1+\epsilon+pq+\binom{q}{2}}g_{i_{-1}i_{-q}}(\alpha_0)_{i_{-q}}\Big[\cdots \Big],\\
			&\color{black}(-1)^{p+q+1+(|\alpha_k|+1)(|\alpha_0|+...+|\alpha_{k-1}|+k-1+q)+\epsilon+pq+\binom{q}{2}}g_{i_0i_{-q}}(\alpha_k\alpha_0)_{i_{-q}}\Big[\cdots \Big]
		\end{split}
	\end{align}
	Next, $d_{Cech}h^{q-1}$ will consist of terms of the form
	\begin{align}
		\begin{split}\label{typ2}
			&\color{black}(-1)^{\epsilon+p(q-1)+\binom{q-1}{2}}g_{i_0i_{1-q}}(\alpha_0)_{i_{1-q}}\Big[(\alpha_1)_{i_{1-q}}\Big|\cdots\Big],\\
			&\color{black}(-1)^{\epsilon+p(q-1)+\binom{q-1}{2}+j}g_{i_0i_{-q}}(\alpha_0)_{i_{-q}}\Big[\cdots\Big|g_{i_{j+1-q}i_{j-1-q}}^{-1}\Big|\cdots\Big],\\
			&\color{black}(-1)^{\epsilon+p(q-1)+\binom{q-1}{2}+q-1+l+1}h^{q-1}(\bar{\alpha})|_{i_{-q}i_{1-q}...i_0...i_{l}j...i_p}
		\end{split}
	\end{align}	
	Also, $h^{q-1}d_{Cech}$ will consist of terms of the form
	\begin{align}
		\begin{split}\label{typ3}
			&\color{black}(-1)^{l+1+\epsilon+(p+1)(q-1)+\binom{q-1}{2}}h^{q-1}(\bar{\alpha}|_{i_0...i_lj...i_p})\\
			&\color{black}(-1)^{\epsilon+(p+1)(q-1)+\binom{q-1}{2}}g_{i_{-1}i_{-q}}(\alpha_0)_{i_{-q}}\Big[\cdots \Big].
		\end{split}
	\end{align}
Finally, $h^q\bar{d}_2$ will contain terms of the form
\begin{align}\begin{split}\label{typ4}
	&\color{black}(-1)^{p+|\alpha_0|+...+|\alpha_r|+r+\epsilon+pq+\binom{q}{2}}g_{i_0i_{-q}}(\alpha_0)_{i_{-q}}\Big[\cdots \Big|(\alpha_{r}\alpha_{r+1})_{j-q}\Big|\cdots\Big]\\
	&\color{black}(-1)^{p+1+(|\alpha_k|+1)(|\alpha_0|+|\alpha_1|+...+|\alpha_{k-1}|+k-1)+\epsilon+pq+\binom{q}{2}}g_{i_0i_{-q}}(\alpha_k\alpha_0)_{i_{-q}}\Big[\cdots \Big]
\end{split}\end{align}
Here the $\epsilon'$ in the signs are all different. The first term in (\ref{typ1}) will cancel with the first term in (\ref{typ2}). The second term in (\ref{typ1}) cancels with the second term in (\ref{typ2}). The third term in (\ref{typ1}) cancels with the first term in (\ref{typ4}). The fourth and fifth terms in (\ref{typ1}) cancel. The sixth term in (\ref{typ1}) cancels with the second term in (\ref{typ3}). The last term in (\ref{typ1}) cancels with the last term in (\ref{typ4}). Finally the last term in (\ref{typ2}) cancels with the first term in (\ref{typ3})

\end{proof}

Now we introduce the second ingredient to the map (\ref{goalisphi}). Suppose $R$ is a commutative ring and $P_0,...,P_m$ are a collection of $\Z/2$-graded free modules. Pick homogeneous bases $\{e_i^j\}$ for each $P_i$. Let $F^0\in \Hom(P_1,P_0)$, $F^1\in \Hom(P_2,P_1)$,..., $F^m\in \Hom(P_0,P_m)$ be endomorphisms which we can think of as matrices with the fixed choice of bases. We define
\begin{align*}
	sTr(F^0\otimes F^1\otimes...\otimes F^m):=\sum(-1)^{\sigma}F^0_{j_0j_1}\otimes F^1_{j_1j_2}\otimes...\otimes F^m_{j_mj_0}
\end{align*}
where the sum is over all $j_0,...,j_m$ and the sign is given by
\begin{align*}
	\sigma=(m+1)|e_0^{j_0}|+|e_1^{j_1}|+...+|e_m^{j_m}|.
\end{align*}
Extending this over infinite sums gives a map of presheaves $\underline{C}^{II}_\bullet(\mathscr{F})\to \underline{C}^{II}_\bullet(\oo_{-f})$ from which we obtain a map $\Cech(\mathscr{U},\underline{C}^{II}_\bullet(\mathscr{F}))\to \Cech(\mathscr{U},\underline{C}^{II}_\bullet(\oo_f))$.

We have to introduce one more bit of notation before defining the map $\phi$. For $\bar{\alpha}:=\alpha_0[\alpha_1|\cdots |\alpha_k]\in C^{II}_\bullet(\mathscr{C}(V))$ let
\begin{align*}
	Sh(\underset{n factors}{\underbrace{[\delta|\cdots|\delta]}},\bar{\alpha})=\sum_{i_0+...+i_k=n} \alpha_0[\delta^{i_0}|\alpha_1|\delta^{i_1}| \alpha_2|\cdots|\alpha_k|\delta^{i_k} ].
\end{align*}
For each $n\geq 0$, this defines a map of presheaves $\underline{C}^{II}_\bullet(\mathscr{C})\to \underline{C}^{II}_\bullet(\mathscr{C})$ and from it we obtain a map $\Cech(\mathscr{U},\underline{C}^{II}_\bullet(\mathscr{C}))\to \Cech(\mathscr{U,}\underline{C}^{II}_\bullet(\mathscr{C}))$.
\begin{deff}
\label{defphi}
We define 
\begin{align*}\begin{split}\phi:\Cech(\mathscr{U},\underline{C}^{II}_\bullet(\mathscr{C}))\to\Cech(\mathscr{U},\underline{C}^{II}_\bullet(\oo_{-f}))
\end{split}\end{align*}
by
\begin{align*}\begin{split}
	\phi=\sum_{n,q\geq 0}(-1)^nsTr\Big(h^q\big(Sh(\underset{n \ factors}{\underbrace{[\delta|\cdots|\delta]}},-)\big)\Big)
\end{split}\end{align*}
\end{deff}

\begin{theorem}\label{theoremonphi}
(i) \ $\phi$ is a chain map. \\
(ii) \ If $(P_i,\delta_i)=(\oo_X,0)$ for some $i$ then $\phi$ is a quasi inverse to the quasi isomorphism (\ref{yon}).
\end{theorem}
\begin{proof}
First note that $\bar{d_0}+\bar{d_2}+d_{Cech}$ commutes with $sTr$. 
It suffices to check that
$$\sum (-1)^nh^{q,n}(d_{Cech}+\bar{d_0}+\bar{d_1}+\bar{d_2})=(d_{Cech}+\bar{d_0}+\bar{d_2})\sum (-1)^nh^{q,n}.$$
Comparing Cech-Hochchild bidegree this equality comes down to 
\begin{align*}
\begin{cases}
(1) \ \ \ (-1)^{n}h^{q,n}\circ \bar{d_1}+\\
(2) \ \ \ (-1)^{n+1}h^{q-1,n+1}\circ d_{Cech}+\\
(3) \ \ \ (-1)^{n-1}h^{q,n-1}\circ \bar{d_0}+\\
(4) \ \ \ (-1)^{n+1}h^{q,n+1}\circ \bar{d_2}
\end{cases}=\begin{cases}
(A) \ \ \ (-1)^{n+1}d_{Cech}\circ h^{q-1,n+1}+\\
(B) \ \ \ (-1)^{n-1}\bar{d_0}\circ h^{q,n-1}+\\
(C) \ \ \ (-1)^{n+1}\bar{d_2}\circ h^{q,n+1}	
\end{cases}
\end{align*}
Then using lemma \ref{lemmaonhq}
\begin{align*}
(A)+(C)=&(-1)^{n+1}d_{Cech}h^{q-1,n+1}(\bar{\alpha})+(-1)^{n+1}\bar{d_2}h^{q,n+1}(\bar{\alpha})	\\
=&(-1)^{n+1}d_{Cech}h^{q-1,0}(Sh(\delta^{n+1},\bar{\alpha}))+(-1)^{n+1}\bar{d_2}h^{q,0}(Sh(\delta^{n+1},\bar{\alpha}))\\
=&(-1)^{n+1}h^{q-1,0}d_{Cech}(Sh(\delta^{n+1},\bar{\alpha}))+(-1)^{n+1}h^{q,0}\bar{d_2}Sh(\delta^{n+1},\bar{\alpha})\\
=&(-1)^{n+1}h^{q-1,n+1}d_{Cech}(\bar{\alpha})+(-1)^nh^{q,n}\bar{d_1}\bar{\alpha}\\
+&(-1)^{n+1+p}h^{q,n-1}Sh(\delta^2,\bar{\alpha})+(-1)^{n+1}h^{q,n+1}\bar{d_2}(\bar{\alpha})\\
=&(2)+(1)+(4)+(3)-(B).
\end{align*}

One can compute that the composite $\underline{\Cech}(\mathscr{U},\underline{C}_\bullet^{II}(\oo_{-f}))\to \underline{\Cech}(\mathscr{U},\underline{C}^{II}_\bullet(\mathscr{C}))\to \underline{\Cech}(\mathscr{U},\underline{C}_\bullet^{II}(\oo_{-f}))\to \underline{\Cech}(\mathscr{U},(\Omega_X^{\bullet},-df\wedge-))$ equals the map $HKR_{(X,-f)}$ from proposition \ref{hkrxf}. From this the second part follows.

\end{proof}

\subsection{Main theorem}\label{maintheorem}
Recall that $Y\subset X$ is locally cut out by $x_i\in \oo_X(U_i)$ and that $u_{ij}x_i=x_j$ over $U_{ij}$. Recall also the matrix factorization $P$ from section \ref{P}.
Over each of the open sets $U_i$ we can identify $P=\oo_X\epsilon_i\oplus \oo_X$. Suppose $F_i$ is a matrix describing the action of an element of $\End(P|_{U_i})$ in this basis. If $F_j$ denotes the matrix of the same endomorphism but with respect to the basis $P=\oo_X\epsilon_j\oplus \oo_X$ over $U_{ji}$ then $g_{ij}F_jg_{ij}^{-1}=F_i$ where the change of basis matrix is 
\begin{align*}
g_{ij}=
\begin{bmatrix}
1&0\\0&u_{ij}
\end{bmatrix}.
\end{align*}

\begin{lemma}\label{diagram1}
The following diagram commutes in $D_{Z/2}(\C)$
$$\begin{tikzcd}
\Cech(\mathscr{U},\underline{C}_\bullet(\cala))\arrow{r}{can}\arrow{dd}{HKR_\cala}&\Cech(\mathscr{U},\underline{C}_\bullet(\EE nd(P^\bullet)))\arrow{r}{\phi}&\Cech(\mathscr{U},\underline{C}_\bullet^{II}(\oo_{-f}))\arrow{dd}{HKR_{\oo_{-f}}}	\\
&&\\
\Cech(\mathscr{U},\Omega_\cala^\bullet)\arrow{r}{\bar{\wedge} -Td(\oo_Y(-Y))^{-1}}&\Cech(\mathscr{U},\Omega_\cala^\bullet)\arrow{r}{\delta}&\Cech(\mathscr{U},\Omega_X^\bullet,-df\wedge)
\end{tikzcd}$$

\end{lemma}

\begin{proof}
Start with an element $a_0[a_1|\cdots|a_k]$ of internal degree zero over $U_{I}$, $I=\{i_0<i_1<...<i_p\}$. Note that the only terms from $\phi$ which do not vanish under $HKR_{\oo_{-f}}$ are those where $n=0$ and $q\geq 1$ or $n=2$ and $q\geq 0$. These are mapped to
\begin{align*}
	\begin{split}
		&-\sum_{q\geq 1} \sum\frac{(-1)^{pq+\binom{q}{2}}}{k!q!}u_{i_0i_{-q}}a_0du_{i_{1-q}i_{-q}}^{-1}\wedge...\wedge du_{i_0i_{-1}}^{-1}\wedge da_1\wedge...\wedge da_k\\
		&-\sum_{q\geq 0} \sum\frac{(-1)^{pq+\binom{q}{2}}}{k!(q+1)!}u_{i_0i_{-q}}a_0du_{i_{1-q}i_{-q}}^{-1}\wedge...\wedge du_{i_0i_{-1}}^{-1}\wedge dx_{i_0}\wedge dg_{i_0}\wedge da_1\wedge...\wedge da_k
	\end{split}
\end{align*}
where the second sums are over all $\{(i_{-q},...,i_{-1})|i_{-q}<i_{1-q}<...<i_{-1}<i_0\}$. This is precisely what we get when going down first and then to the right twice.

Now let's start with an element of with precisely one factor from $\cala^{-1}$, $\bar{b}=b_0[b_1|\cdots|b_l\epsilon_{i_0}|\cdots |b_k]$ and first go to the right twice and then down. This time, the only terms in the definition of $\phi$ which don't vanish under $HKR_{\oo_{-f}}$ is $n=1$ and $q\geq 0$. We get

\begin{align*}
	\begin{split}
		\sum \frac{(-1)^{1+l+pq+\binom{q}{2}}}{k!(q+1)!}u_{i_0i_{-q}}b_0du_{i_{1-q}i_{-q}}^{-1}\wedge...\wedge du_{i_0i_{-1}}^{-1}\wedge dx_{i_0}\wedge db_1\wedge...\wedge db_k
	\end{split}
\end{align*}
which agrees with what we get going down and then to the right.

Finally if we start with an element with two or more factors from $\cala^{-1}$, then going either direction is zero. this is clear if we go down first because $HKR_\cala$ vanishes on such elements by definition. For the other direction note that the only terms in the definition of $\phi$ which are non-zero on such an element are the ones where $n$ equals the number of factors from $\cala^{-1}$ so $n\geq 2$. Each such term will contain two or more factors $x_i$ and will therefore vanish when we apply $HKR_{\oo_{-f}}$ because $dx_i\wedge dx_i=0$.
\end{proof}

\begin{lemma}\label{diagram2}
There is a commutative diagram in $D_{Z/2}(\C)$
$$\begin{tikzcd}
	C_\bullet(\text{perf}_{\Cech}(Y))\arrow{d}{\sim}&C_\bullet(\text{perf}_{\Cech}(\cala))\arrow{l}{\sim}\arrow{d}{\sim}\\
	\Cech(\mathscr{U},\Omega_Y^\bullet)&\Cech(\mathscr{U},\Omega_\cala^\bullet)\arrow{l}{\sim}
\end{tikzcd}$$
\end{lemma}

\begin{proof}
	We will divide the diagram into smaller commutative pieces where each arrow is a chain map
	
$$\begin{tikzcd}
				C_\bullet(\text{\text{perf}}_{\Cech}(Y))\arrow{d}{}&C_\bullet(\text{\text{perf}}_{\Cech}(\cala))\arrow{d}{}\arrow{l}{}\arrow[phantom, shift left, shift left, shift left, shift left, shift left, shift left, shift left, shift left]{l}{(A)}\\
		\Cech(\mathscr{U},\underline{C}_\bullet(\underline{\text{\text{perf}}}_{\Cech}(Y)))&\Cech(\mathscr{U},\underline{C}_\bullet(\underline{\text{\text{perf}}}_{\Cech}(\cala)))\arrow{l}{}\arrow{l}{}\arrow[phantom, shift left, shift left, shift left, shift left, shift left, shift left, shift left, shift left]{l}{(B)}\\
		\Cech(\mathscr{U},\underline{C}_\bullet(\oo_Y))\arrow{u}{}\arrow{d}{}&\Cech(\mathscr{U},\underline{C}_\bullet(\cala))\arrow{u}{}\arrow{l}{}\arrow[phantom, shift left, shift left, shift left, shift left, shift left, shift left, shift left, shift left, shift left]{l}{(C)}\arrow{d}{}\\
		\Cech(\mathscr{U},\Omega_Y^\bullet)&\Cech(\mathscr{U},\Omega_\cala^\bullet)\arrow{l}{}\arrow{l}{}\\
	\end{tikzcd}$$

Note first that all the vertical arrows on the left are quasi isomorphisms by the discussion in section \ref{thecasef0}. We will see below that the horizontal arrows in the squares $(A)-(C)$ are quasi isomorphisms and then it follows from commutativity (which is also dealt with below) that the vertical arrows on the right too are quasi isomorphisms.

In the square labeled $(A)$ the top arrow is a quasi isomorphism by lemma \ref{indfunctor}. The bottom arrow in this square is $\Cech(\text{Ind},\alpha)$ where $(\text{Ind},\alpha)$ is the lax morphism of presheaves of dg categories from proposition \ref{laxind}. Commutativity of $(A)$ is precisely proposition \ref{proponlaxmorphisms}. The bottom arrow in this square can be shown to be a quasi isomorphism by filtering both complexes by Cech degree and using the mapping theorem (\cite{Macl},theorem XI.3.4).

To see that the square $(B)$ commutes, we first note that we can assume that in the diagram below, we have $(\text{Ind},\alpha)\circ Yoneda=(Yoneda\circ q,1)$ as lax morphisms of presheaves of dg categories
$$\begin{tikzcd}\underline{\text{perf}}_{\Cech}(Y)&\underline{\text{perf}}_{\Cech}(\cala)\arrow{l}{\varphi}\\
\oo_Y\arrow{u}{Yoneda}&\cala\arrow{u}{Yoneda}\arrow{l}{q}.
\end{tikzcd}$$
Commutativity then follows from lemma \ref{comparinglaxandstrict}.

Commutativity of the square $(C)$ is part of proposition \ref{proponhkra}.

\end{proof}

\begin{lemma}\label{upperdiagram}
There is a commutative diagram in $D_{Z/2}(\C)$

\begin{tikzcd}
	C_\bullet(\text{\text{perf}}_{\Cech}(\cala))\arrow{d}{\sim}\arrow{rr}{}&&C_\bullet(\text{vect}_{\Cech}(X,f))\arrow{d}{\sim}\\
		\Cech(\mathscr{U},\underline{C}_\bullet(\cala))\arrow{r}{can}&\Cech(\mathscr{U},\underline{C}_\bullet(\underline{\End}(P)))\arrow{r}{\phi}&\Cech(\mathscr{U},\underline{C}_\bullet^{II}(\oo_{-f}))
\end{tikzcd}
\end{lemma}

\begin{proof}
	We will divide the diagram into smaller commutative pieces where each arrow is a chain map:
	
	\begin{tikzcd}
	C_\bullet(\text{\text{perf}}_{\Cech}(\cala))\arrow{d}{\sim}\arrow{rr}{}\arrow[phantom, shift right, shift right, shift right, shift right, shift right, shift right, shift right, shift right, shift right]{rr}{(A)}&&C_\bullet(\text{vect}_{\Cech}(X,f))\arrow{d}{\sim}\\
	\Cech(\mathscr{U},\underline{C}_\bullet(\underline{\text{perf}}_{\Cech}(\cala)))\arrow{rr}{}\arrow[phantom, shift right, shift right, shift right, shift right, shift right, shift right, shift right, shift right, shift right]{rr}{(B)}&&\Cech(\mathscr{U},\underline{C}_\bullet(\underline{\text{vect}}_{\Cech}(X,f)))\\
	\Cech(\mathscr{U},\underline{C}_\bullet(\cala))\arrow{u}{\sim}\arrow{rr}{}\arrow[phantom, shift right, shift right, shift right, shift right, shift right, shift right, shift right]{rr}{(C)}&&\Cech(\mathscr{U},\underline{C}_\bullet(\underline{\text{vect}}(X,f)))\arrow{u}{\sim}\arrow{d}{\sim}\\
	\Cech(\mathscr{U},\underline{C}_\bullet(\cala))\arrow{d}{=}\arrow{u}{=}\arrow{r}{}\arrow[phantom, shift left, shift left, shift left, shift left, shift left, shift left, shift left, shift left, shift left, shift left, shift left, shift left, shift left, shift left, shift left, shift left, shift left, shift left,shift left, shift left, shift left, shift left, shift left]{d}{(D)}&\Cech(\mathscr{U},\underline{C}^{II}_\bullet(\mathscr{C}))\arrow{r}{}\arrow[shift left]{dr}{\phi}&\Cech(\mathscr{U},\underline{C}^{II}_\bullet(\underline{\text{qvect}}(X,f)))\\
		\Cech(\mathscr{U},\underline{C}_\bullet(\cala))\arrow{r}{}&\Cech(\mathscr{U},\underline{C}_\bullet(\mathcal{E}nd(P))\arrow{u}{}\arrow{r}{\phi}\arrow[phantom, shift right, shift right, shift right, shift right, shift right, shift right, shift right]{u}{(E)}&\Cech(\mathscr{U},\underline{C}^{II}_\bullet(\oo_{-f}))\arrow[shift left]{ul}{}\arrow{u}{\sim}\arrow[phantom, shift left, shift left, shift left, shift left, shift left, shift left, shift left, shift left, shift left, shift left]{u}{(F)}
	\end{tikzcd}
	We already know that the vertical arrows on the left and right are quasi isomorphisms. Indeed, on the left hand side this was dealt with in the proof of the previous lemma. Those on the right hand side were dealt with in section  \ref{secthochofmat}. 
	
	The bottom arrow of the square $(A)$ is $\Cech(-\otimes_\cala P,\beta)$ where $(-\otimes_\cala P,\beta)$ is the lax morphisms of presheaves of dg categories from proposition \ref{laxind}. Commutativity of the square (A) is exactly proposition \ref{proponlaxmorphisms}. 
Commutativity of (B) is similar to commutativity of the square (B) in the previous lemma. Note that all arrows except the top horizontal one are induced by morphisms of presheaves of dg categories and the top map is induced by a lax morphism. In the square $(C)$, $\mathscr{C}$ denotes the presheaf of cdg categories $\langle P,\oo_X\rangle\subset \underline{\text{qvect}}(X,f)$. All the arrows in $(C)$ are induced by morphisms of presheaves of cdg categories and the diagram commutes already at the level of presheaves of cdg categories. Same holds for the square (D). In the triangle (F) all the arrows are quasi isomorphisms and the two diagonal arrows are quasi inverses to each other. If we choose the diagonal arrow going up then (F) clearly commutes because it comes from a commutative diagram of presheaves of cdg categories and morphisms of such. Finally the triangle (E) commutes if we choose the diagonal arrow going down by the definition of the map $\phi$ from section \ref{constructingphi}
\end{proof}
Now we are ready to prove that the map on Hochschild homology induced by the pushforward functor $i_*:D^b(\text{coh}Y)\to D(\text{fact}(X,f))$ is given by first multiplying by the inverse Todd class described in section \ref{inversetodd} and then applying the connecting homomorphism $\delta$ coming from the short exact sequence (\ref{SES}). We formulate this in a precise way in the following theorem (which is the same as theorem 1 from the introduction).
\begin{theorem}
	There is a commutative diagram in $D_{Z/2}(\C)$
	$$\begin{tikzcd}
		HH_\bullet(Y)\arrow{d}{\sim}\arrow{rrr}{HH_\bullet(i_*)}&&&HH_\bullet(X,f)\arrow{d}{\sim}\\
		\Cech(\mathscr{U},\Omega_Y^\bullet)\arrow{rr}{\bar{\wedge}Td(-Y)^{-1}}&&\Cech(\mathscr{U},\Omega_Y^\bullet)\arrow{r}{\delta}&\Cech(\mathscr{U},(\Omega_X^\bullet,-df\wedge-)).
	\end{tikzcd}$$
\end{theorem}

\begin{proof}
Again we break it into smaller commutative pieces
$$\begin{tikzcd}
	&&C_\bullet^{II}(\text{vect}(X,f))\arrow{dddd}{}\\
	C_\bullet(\text{perf}_{\Cech}(Y))\arrow{d}{\sim}\arrow{urr}{}&C_\bullet(\text{perf}_{\Cech}(\cala))\arrow{l}{\sim}\arrow{ur}{}\arrow{d}{\sim}&\\
	\Cech(\mathscr{U},\Omega_Y^\bullet)\arrow{d}{\bar{\wedge}Td(-Y)^{-1}}&\Cech(\mathscr{U},\Omega_\cala^\bullet)\arrow{l}{\sim}\arrow{d}{\bar{\wedge}Td(-Y)^{-1}}&\\
	\Cech(\mathscr{U},\Omega_Y^\bullet)\arrow{drr}{\delta}&\Cech(\mathscr{U},\Omega_\cala^\bullet)\arrow{l}{\sim}\arrow{dr}{}&\\
	&&\Cech(\mathscr{U},(\Omega_X^\bullet,-df\wedge))
\end{tikzcd}
$$
The top left triangle commutes by lemma \ref{resolutionforoy}. The top left square commutes by lemma \ref{diagram2}. The bottom left square commutes by lemma \ref{propontodd}. The bottom left triangle commutes by definition of the connecting morphism $\delta$. The part on the right commutes by lemmas \ref{diagram1} and \ref{upperdiagram}.
\end{proof}

\end{document}